\documentclass[a4paper,12pt]{article}

\usepackage[utf8]{inputenc}
\usepackage[T1]{fontenc}
\usepackage{graphicx}
\usepackage{longtable}
\usepackage{float}
\usepackage{wrapfig}
\usepackage{rotating}
\usepackage[normalem]{ulem}
\usepackage{amsmath}
\usepackage{amsthm}
\usepackage{textcomp}
\usepackage{marvosym}
\usepackage{wasysym}
\usepackage{amssymb}
\usepackage{capt-of}
\usepackage{hyperref}
\usepackage{comment}
\usepackage{color}
\usepackage{subfigure}
\def\sizefig{0.65}
\def\sizesmallfig{0.29}

\def\S{\mathbf{S}}
\def\A{\mathbf{A}}
\def\B{\mathbf{B}}

\def\g{\mathbf{g}}
\def\b{\mathbf{b}}

\def\Cb{\mathbf{C}}
\def\G{\mathbf{G}}
\def\h{\mathbf{h}}
\def\S{\mathbf{S}}

\def\X{\mathbf{X}}
\def\T{\mathbf{T}}

\def\x{\mathbf{x}}

\def\y{\mathbf{y}}
\def\p{\mathbf{p}}
\def\z{\mathbf{z}}
\renewcommand\u{\mathbf{u}}
\renewcommand\v{\mathbf{v}}
\def\M{\mathbf{M}}
\def\Lop{\mathcal{L}}
\def\sing{\text{sing}}
\def\ac{\text{ac}}
\def\tv{\text{TV}}
\def\cont{\text{cont}}
\def\disc{\text{disc}}
\usepackage{mathrsfs}
\def\linfun{\mathscr{I}}
\newcommand{\R}{\mathbb{R}}
\newcommand{\C}{\mathbb{C}}
\newcommand{\N}{\mathbb{N}}

\def\div{\text{div}}
\def\grad{\text{grad}}

\theoremstyle{plain}
\newtheorem{theorem}{Theorem}[section]
\newtheorem{lemma}[theorem]{Lemma}

\theoremstyle{definition}
\newtheorem{definition}[theorem]{Definition}
\newtheorem{assumption}[theorem]{Assumption}

\newtheorem{remark}{Remark}

\if{
\newcommand{\qed}{\nobreak \ifvmode \relax \else
      \ifdim\lastskip<1.5em \hskip-\lastskip
      \hskip1.5em plus0em minus0.5em \fi \nobreak
      \vrule height0.75em width0.5em depth0.25em\fi}
}\fi

\DeclareMathOperator{\vol}{vol}
\DeclareMathOperator{\diam}{diam}

\DeclareMathOperator{\dist}{dist}

\usepackage[normalem]{ulem}
\usepackage{color}

\begin{document}
\author{Victor Magron$^{1}$ \and Marcelo Forets$^{1}$ \and Didier Henrion$^{2,3}$}
\date{\today}
\title{Semidefinite Approximations of Invariant Measures for  Polynomial Systems}

\footnotetext[1]{CNRS; VERIMAG; 700 av Centrale, 38401 Saint-Martin d'Hères; France}
\footnotetext[2]{CNRS; LAAS; Universit\'e de Toulouse; 7 avenue du colonel Roche, F-31400 Toulouse; France}
\footnotetext[3]{Faculty of Electrical Engineering, Czech Technical University in Prague, Technick\'a 4, CZ-16206 Prague, Czechia}
\maketitle



\begin{abstract}

We consider the problem of approximating numerically the moments and the supports of measures which are invariant with respect to the dynamics of continuous- and discrete-time polynomial systems, under semialgebraic set constraints. 
First, we address the problem of approximating the density and hence the support of an invariant measure which is absolutely continuous with respect to the Lebesgue measure. Then, we focus on the  approximation of the support of  an invariant measure which is singular with respect to the Lebesgue measure.

Each problem is handled through an appropriate reformulation into a linear optimization problem over measures, solved in practice with two hierarchies of finite-dimensional semidefinite moment-sum-of-square relaxations, also called Lasserre hierarchies.

Under specific assumptions, the first Lasserre hierarchy allows to approximate the moments of an absolutely continuous invariant measure as close as desired and to extract a sequence of polynomials converging weakly to the density of this measure.

The second Lasserre hierarchy allows to approximate as close as desired in the Hausdorff metric the support of a singular invariant measure with the level sets of the Christoffel polynomials associated to the moment matrices of this measure.

We also present some application examples together with numerical results for several dynamical systems admitting either absolutely continuous  or singular invariant measures. 

\end{abstract}


\section{Introduction}
\label{sec:intro}

Given a polynomial system described by a discrete-time (difference) or continuous-time (differential) equation under general semialgebraic constraints, we propose {numerical methods to approximate the moments and the supports} of the measures which are invariant under the sytem dynamics. 
The characterization of invariant measures allows to determine important features of long term dynamical behaviors \cite{LM94}.

One classical way to approximate such features is to perform numerical integration of the equation satisfied by the system state after choosing some initial conditions. However, the resulting trajectory could exhibit some chaotic behaviors or great sensitivity with respect to the initial conditions.

Numerical computation of invariant sets and measures of dynamical systems have previously been studied using domain subdivision techniques, where the density of an invariant measure is recovered as the solution of fixed point equations of the discretized Perron-Frobenius operator \cite{dellnitz1997exploring,aston2014computing}. The underlying  method, integrated in the software package GAIO \cite{dellnitz2001algorithms}, consists of covering the invariant set by boxes and then approximating the dynamical behaviour by a Markov chain based on transition probabilities between elements of this covering. More recently, in \cite{dellnitz2017set} the authors have developed a multilevel subdivision scheme that can handle uncertain ordinary differential equations as well.

By contrast with most of the existing work in the literature, our method does not rely neither on time nor on space discretization. In our approach, the invariant measures are modeled with finitely many moments, leading to approximate recovery of densities in the absolutely continuous case or supports in the singular case. Our contribution is in the research trend aiming at characterizing the behavior of dynamical nonlinear systems through linear programs (LP), whose unkown are measures supported on the system constraints. This methodology was introduced in the static case by Lasserre~\cite{Lasserre01moments} and consists of reformulating a polynomial optimization problem as an infinite-dimensional LP over probability measures. To handle practically such an LP, one can then rely on a hierarchy of semidefinite programming (SDP) problems, called moment-sum-of-square or Lasserre hierarchy, see \cite{lasserre2009moments} for a global view.

In the context of polynomial dynamical systems, extensions to this hierarchy have been studied, for example, to obtain converging approximations for regions of attraction~\cite{HK14roa}, maximum controlled invariants~\cite{KHJ13mci} and reachable sets~\cite{MGHT17reach}.

In our case, we first focus on the characterization of  densities of absolutely continuous invariant measures with respect to some reference measure (for instance the Lebesgue measure). For this first problem, our method is inspired by previous contributions looking for moment conditions ensuring that the underlying unknown measure is absolutely continuous~\cite{lasserre2013borel} with a bounded density in a Lebesgue space, as a follow-up of the volume approxilation results of \cite{HLS09vol}. When the density function is assumed to be square-integrable, one can rely on~\cite{HDM14meansquared} to build a sequence of polynomial approximations converging to this density in the $L^2$-norm.

We focus later on the characterization of 
supports of singular invariant measures. For this second problem, we rely on previous works~\cite{HLS09vol,Las16Decomp,Christoffel17} aiming at extracting as much information as possible on the support of a  measure from the knowledge of its moments. 
{The numerical scheme proposed in~\cite{HLS09vol} allows to approximate as close as desired the moments of a measure uniformly supported on a given semialgebraic set.
The framework from~\cite{Las16Decomp} uses similar techniques to compute effectively the Lebesgue decomposition of a given measure}, while~\cite{Christoffel17} relies on the Christoffel function associated to the moment matrix of this measure. When the measure is uniform or when the support of the measure satisfies certain conditions, the sequence of level sets of the Christoffel function converges to the measure support with respect to the Hausdorff distance.

Previous work by the third author~\cite{HenrionFixpoints} shows how to use the Lasserre hierarchy to characterize invariant measures for one-dimensional discrete polynomial dynamical systems. We extend significantly this work in the sense that we now characterize invariant measures on more general multidimensional semialgebraic sets, in both discrete and continuous settings, and we establish convergence guarantees under certain assumptions.
In the concurrent work \cite{Korda}, the authors are also using the Lasserre hiearchy for approximately computing extremal measures, i.e. invariant measures optimal w.r.t.~a convex criterion. They have weaker convergence guarantees than ours, but the problem is formulated in a more general set-up including physical measures, ergodic measures or atomic measures.

Our contribution is twofold:

\begin{itemize}
\item A first Lasserre hierarchy allowing to approximate the moments of an invariant measure which is absolutely continuous with respect to the Lebesgue measure. For an invariant measure with either square integrable or essentially bounded density, one has  convergence guarantees of the hierarchy and one can compute asymptotically the exact moment sequence of this measure.
At each step of the hierarchy, one can recover an approximate polynomial density from the solution of the SDP problem. 
The resulting sequence of polynomial approximations converges weakly to the exact density when the degree of the polynomial goes to infinity.
\item A second Lasserre hierarchy allowing to approximate as close as desired the support of a singular invariant measure. At each step of this hierarchy, one can recover an approximation of the support with the superlevel set of the Christoffel polynomial associated to the moment matrix of the singular measure. Under certain assumptions, the maximal distance between the exact support and the points among this superlevel set converges to zero when the size of the moment matrix goes to infinity, i.e. we can ensure convergence in the Hausdorff metric.
\end{itemize}

In both cases, our results apply for both discrete-time and continuous-time polynomial systems. Each problem is handled through an adequate reformulation into an LP problem over probability measures. We show how to solve in practice this infinite-dimensional problem with moment relaxations. In both cases, this boils down to solving a Lasserre hierarchy of finite-dimensional SDP problems of increasing size.
  
In Section~\ref{sec:pb}, we describe preliminary materials about discrete/continuous-time polynomial systems, invariant measures as well as polynomial sums of squares and moment matrices. 
The first problem of density approximation for absolutely continuous invariant measures is handled in Section~\ref{sec:ac} with our first hierarchy. The second problem of support approximation for singular invariant measures is investigated in Section~\ref{sec:sing} with our second hierarchy.
Finally, we illustrate the method with several numerical experiments in Section~\ref{sec:bench}.

\section{Invariant Measures and Polynomial Systems}
\label{sec:pb}

\subsection{Discrete-Time and Continuous-Time Polynomial Systems}
\label{sec:systems}

Given $r,n \in \N$, let $\R[\x]$ (resp.~$\R_{2r}[\x]$) stand for the vector space of real-valued $n$-variate polynomials (resp. of degree at most $2r$) in the variable $\x=(x_1,\ldots,x_n) \in \R^n$. Let $\C[\x]$ be the vector space of complex-valued $n$-variate polynomials.
We are interested in the polynomial system defined by

\begin{itemize}
\item a polynomial transition map 
\begin{equation}
\label{eq:defPolynomialMap_f}
 f : \R^n \to \R^n,\qquad \x \mapsto f(\x) := (f_1(\x), \dots, f_n(\x)) \in \R^n[\x]
\end{equation}
with $f_1,\dots,f_n \in \R[\x]$. The degree of $f$ is given by $d_f := \max \{\deg f_1,\ldots,\deg f_n\}$;
\item a set of constraints assumed to be compact basic semialgebraic:
\begin{equation} 
\label{eq:defX}
\X :=  \{\x \in \R^n : g_1(\x)  \geq 0, \dots, g_m (\x) \geq 0 \} \,,
\end{equation} 
defined by given polynomials $g_1,\dots,g_m \in \R[\x]$.
\end{itemize}

We develop our approach in parallel for discrete-time and continuous-time systems. 
With $f$ being a polynomial transition as in~\eqref{eq:defPolynomialMap_f} and $\X$ being a set of semialgebraic state constraints as in~\eqref{eq:defX}, we consider either the discrete-time system:
\begin{align}
\label{eq:disc}
\x_{t+1} = f(\x_t) \,, \quad \x_t \in \X \,, \quad t \in \N \,,
\end{align}
or the continuous-time system:
\begin{align}
\label{eq:cont}
\dot{\x}(t) = \frac{d\x(t)}{dt} = f(\x(t)) \,, \quad \x(t) \in \X \,, \quad t \in [0, \infty) \,.
\end{align}

\subsection{Invariant Measures}
\label{sec:invariant}

Given a compact set $\A \subset \R^n$, we denote by $\mathcal{M}(\A)$ the vector space of finite signed Borel measures supported on $\A$, namely real-valued functions on the Borel sigma algebra $\mathcal{B}(\A)$. 
The support of a measure $\mu \in \mathcal{M}(\A)$ is defined as the closure of the set of all points $\x$ such that $\mu(\B) \neq 0$ for an open neighborhood $\B$ of $\x$.
We note $\mathcal{C}(\A)$ (resp.~$\mathcal{C}^1(\A)$) the Banach space of continuous (resp.~continuously differentiable) functions on $\A$ equipped with the sup-norm.
Let $\mathcal{C}(\A)'$ be the topological dual of $\mathcal{C}(\A)$ (equipped with the sup-norm), i.e.~the set of continuous linear functionals of $\mathcal{C}(\A)$.
By a Riesz representation theorem, $\mathcal{C}(\A)'$ is isomorphically identified with
$\mathcal{M}(\A)$ equipped with the total variation norm denoted by $\|
\cdot\|_{\text{TV}}$.
Let $\mathcal{C}_+(\A)$ (resp.~$\mathcal{M}_+(\A)$) be the cone of non-negative elements of $\mathcal{C}(\A)$ (resp. $\mathcal{M}(\A)$). 
{A probability measure on $\A$ is an element $\mu \in \mathcal{M}_+(\A)$ such that $\mu(\A)=1$.}
The topology in $\mathcal{C}_+(\A)$ is the strong topology of uniform convergence in contrast with the weak-star topology that can be defined in $\mathcal{M}_+(\A)$. 

The restriction of the {Lebesgue measure} on a subset $\A \subseteq \X$ is
$\lambda_\A (d \x) := \mathbf{1}_\A(\x) \, d \x $, 
where $\mathbf{1}_\A : \X \to \{0, 1\}$ stands for the {indicator function} of $\A$, namely $\mathbf{1}_\A(\x) = 1$ if $\x \in \A$ and
 $\mathbf{1}_\A(\x) = 0$ otherwise.

The moments of the Lebesgue measure on $\A$ are denoted by
\begin{equation}
\label{momb}
z^\A_{\beta} := \int \x^{\beta} \lambda_\A(d\x) \in \R \,, \quad \beta \in \N^n
\end{equation}
where we use the multinomial notation $\x^{\beta} := x^{\beta_1}_1 x^{\beta_2}_2 \ldots x^{\beta_n}_n$. 
The sequence of Lebesgue moments on $\A$ is denoted by $\z^\A:=(z^\A_{\beta})_{\beta \in \N^n}$.
The Lebesgue volume of $\A$ is $\vol \A := z^\A_0 = \int \lambda_\A (d\x)$. When $\A = \X$, we define $\z := \z^\X$ by omitting the superscript notation. A sequence $\y:=(y_\beta)_{\beta \in \N^n} \in \R^{\N^n}$ is said to have a {representing measure} on $\X$ if there exists $\mu \in \mathcal{M}(\X)$ such that $y_\beta = \int \x^\beta \mu(d\x)$ for all $\beta \in \N^n$.

Given  $\mu,\nu \in \mathcal{M}(\A)$,
the notation
\[
\mu \leq \nu
\]
means that $\nu-\mu \in \mathcal{M}_+(\A)$, namely
that $\mu$ is {dominated} by $\nu$.

Given  $\mu \in \mathcal{M}_+(\A)$, there exists a unique Lebesgue decomposition $\mu = \nu + \psi$ with $\nu, \psi \in  \mathcal{M}_+(\A)$, $\nu \ll \lambda$ and $\psi \perp \lambda$. Here, the notation $\nu \ll \lambda$ means that $\nu$ is {absolutely continuous} with respect to (w.r.t.) $\lambda$, {that is, for every  $\A \in \mathcal{B}(\X)$, $\lambda (\A) = 0$ implies $\nu(\A) = 0$.}
The notation $\psi \perp \lambda$ means that $\psi$ is {singular} w.r.t.~$\lambda$, {that is, there exist disjoint sets $\A, \B \in \mathcal{B}(\X)$ such that $\A \cup \B = \X$ and $\psi(\A) = \lambda(\B) = 0$.}

The so-called pushforward measure or {image measure}  of a given $\mu \in \mathcal{M}_+(\X)$ under $f$ is defined as follows:
\begin{equation}
\label{eq:invmeas}
f_\# \mu (\A) := \mu (f^{-1}(\A)) = \mu (\{ \x \in \X : f(\x) \in \A \}),
\end{equation}
for every set $\A \in \mathcal{B}(\X)$.

See~\cite[Section 21.7]{Royden} and \cite[Chapter IV]{alexander2002course} or~\cite[Section 5.10]{Luenberger97} for additional background on functional analysis, measure theory and applications in convex optimization. For more details on image measures, see e.g.~\cite[Section 1.5]{AFP00}.

Let us define the linear operator $\Lop_f^\disc : \mathcal{C}(\X)' \to \mathcal{C}(\X)'$ by: 
\[
\Lop_f^\disc (\mu) :=  f_{\#} \mu-\mu
\]
and the linear operator $\Lop_f^\cont: \mathcal{C}^1(\X)' \to \mathcal{C}(\X)'$  by:
\[
\Lop_f^\cont (\mu) :=  \div(f\mu) = \sum_{i=1}^n \frac{\partial (f_i \mu)}{\partial x_i}
\]
where the derivatives of measures are understood in the sense of distributions, that is, through their action on test functions of $\mathcal{C}^1(\X)$. 
In the sequel, we use the more concise notation $\Lop_f$ to refer to $\Lop_f^\disc$ (resp.~$\Lop_f^\cont$) in the context of discrete-time (resp.~continuous-time) systems. 

\begin{definition}{\textbf{(Invariant measure)}}
\label{eq:inv}
We say that a measure $\mu$ is invariant w.r.t.~$f$ when $\Lop_f(\mu) = 0$ and refer to such a measure as an {invariant measure}. We also omit the reference to the map $f$ when it is obvious from the context and write $\Lop(\mu) = 0$.
\end{definition}

When considering discrete-time systems as in~\eqref{eq:disc}, a measure $\mu$ is called  invariant w.r.t.~$f$ when it satisfies $\Lop_f^\disc (\mu) = 0$. 
When considering continuous-time systems as in~\eqref{eq:cont}, a measure is called invariant w.r.t.~$f$ when it satisfies $\Lop_f^\cont (\mu) = 0$. 

{It was proved in~\cite{Krylov37} that a continuous map of a compact metric space into itself has at least one invariant probability measure. A probability measure $\mu$ is ergodic w.r.t.~$f$ if for all $\A \in \mathcal{B}(\X)$ with $f^{-1}(\A) = \A$, one has either $\mu(\A) = 0$ or $\mu(\A) = 1$. The set of invariant probability measures is a convex set and the extreme points of this set consist of the so-called {ergodic measures}. 
}
For more material on dynamical systems and invariant measures, we refer the interested reader to~\cite{LM94}.  

\subsection{Sums of Squares and Moment Matrices}
\label{sec:sdp}

With $\X$ a basic compact semialgebraic set as in~\eqref{eq:defX}, we set $r_j := \lceil (\deg g_j ) / 2 \rceil, j = 1, \dots, m$. For the ease of further notation, we set $g_0(\x) := 1$.
Let $\Sigma[\x]$ stand for the cone of polynomials that can be expressed as sums of squares (SOS) of some polynomials, and let us note $\Sigma_r[\x]$ the cone of SOS polynomials of degree at most $2 r$, namely $\Sigma_r[\x] := \Sigma[\x] \cap \R_{2r}[\x]$.

For each $r \geq r_{\min} := \max \{1, r_1, \dots,r_m \}$, let ${\mathcal Q}_r(\X)$ be the $r$-truncated quadratic module generated by $g_0, \dots, g_m$:
\begin{align*}
{\mathcal Q}_r(\X) & := \Bigl\{\,\sum_{j=0}^{m} s_j(\x) {g_j} (\x) : s_j \in \Sigma_{r - r_j}[\x], \,  j = 0, \dots, m  \,\Bigr\},
\end{align*}
a convex cone of $\R_{2r}[\x]$.
Let ${\mathcal P}(\X):=\R[\x]\cap{\mathcal C_+}(\X)$ denote the cone of polynomials that are nonnegative on $\X$.
To guarantee the convergence behavior of the relaxations presented in the sequel, we need to ensure that polynomials which are in the interior of ${\mathcal P}(\X)$ lie in ${\mathcal Q}_r(\X)$ for some $r \in \N$. The existence of such SOS-based representations is guaranteed by Putinar's Positivstellensaz (see e.g.~\cite[Theorem~2.2]{HLS09vol}), when the following algebraic compactness condition holds:
\begin{assumption}
\label{hyp:archimedean}
There exists a large enough integer $N$ such that one of the polynomials describing the set $\X$ is equal to $N - \| \x \|_2^2$.
\end{assumption}

In addition, the semialgebraic set $\X$ 
should fulfill the following condition: 
\begin{assumption}
\label{hyp:momb}
The moments~\eqref{momb} of the Lebesgue measure on $\X$ are available analytically.
\end{assumption}
This is the case e.g. if $\X$ is a ball.
From now on we assume that $\X$ is a compact basic semialgebraic set as in~\eqref{eq:defX} and that it satisfies both Assumptions \ref{hyp:archimedean} and \ref{hyp:momb}.

For all $r \in \N$, we set $\N^{n}_r := \{ \beta \in \N^{n} : \sum_{j=1}^{n} \beta_j \leq r \}$, whose cardinality is $\binom{n+r}{r}$.
Then a polynomial $p \in \R_r[\x]$ is written as follows:
\[\x \mapsto p(\x) \,=\,\sum_{\beta \in\N^n_r} \, p_{\beta} \, \x^\beta \:, \]
and $p$ is identified with its vector of coefficients $\p=(p_{\beta})_{\beta\in\N^n_r}$ in the canonical basis $(\x^\beta)_{\beta \in\N^n_r}$.

Given a real sequence $\y =(y_{\beta})_{\beta \in \N^n} \in \R^{\N^n}$, let us define the linear functional $\ell_\y : \R[\x] \to \R$ by $\ell_\y(p) := \sum_{\beta} p_{\beta} y_{\beta}$, for every polynomial $p$. 
For each $j=0,1,\ldots,m$, we associate to $\y$ 
a {localizing matrix}, that is a real symmetric matrix $\M_r(g_j \, \y)$ with rows and columns indexed by $\N_{r-r_j}^{n}$ and the following entrywise definition: 
\[ 
(\M_r(g_j \, \y))_{\beta, \gamma} := \ell_\y(g_j(\x) \, \x^{\beta + \gamma}) \,, \quad
\forall \beta, \gamma \in \N_{r-r_j}^n \,. 
\]
When $j=0$ the localizing matrix is called the moment matrix $\M_r(\y):=\M_{r-r_0}(g_0\,\y)$.

For a given invariant measure $\mu \in \mathcal{M}(\X)$, one has $\Lop(\mu) = 0$. 
{It follows from the Stone-Weierstrass Theorem that monomials are dense in continuous functions on the compact set $\X$. The equation $\Lop(\mu) = 0$ is then equivalent to}
\[
\ell_\y (f(\x)^\alpha) - \ell_\y (\x^\alpha) = 0 \,, \quad \forall \alpha \in \N^n \,,
\]
in the context of a discrete-time system~\eqref{eq:disc} and 
\[
\sum_{i=1}^n \ell_\y \biggl(\frac{\partial (\x^\alpha)}{\partial x_i} f_i(\x) \biggr) = 0 \,, \quad \forall \alpha \in \N^n \,,
\]
in the context of a continuous-time system~\eqref{eq:cont}.

Hence, we introduce the linear functionals $\linfun_\y^\disc: \R[\x] \to \R$ defined by
\[
\linfun_\y^\disc (p) := \ell_\y (p \, \circ f) - \ell_\y (p)
\]
and  $\linfun_\y^\cont : \R[\x] \to \R$ defined by
\[
\linfun_\y^\cont (p) := \ell_\y (\grad\:p \cdot f)
\]
for every polynomial $p$ and where $\grad\:p:=(\frac{\partial p}{\partial x_i})_{i=1,\ldots,n}$.
In the sequel, we use the more concise notation $\linfun_\y$ to refer to $\linfun_\y^\disc$ (resp.~$\linfun_\y^\cont$) in the context of discrete-time (resp.~continuous-time) systems. 

\section{Absolutely Continuous Invariant Measures} 
\label{sec:ac}

For $p=1,2,\ldots$, let $L^p(\X)$ (resp.~$L^p_+(\X)$) be the space of (resp.~nonnegative) Lebesgue integrable functions $f$ on $\X$, i..e. such that $\|f\|_p:=(\int_{\X} \vert f(\x) \vert^p \lambda(d\x))^{1/p} < \infty$.
Let  $L^\infty(\X)$ (resp.~$L^\infty_+(\X)$)  be the space of (resp.~nonnegative) Lebesgue integrable functions $f$ on $\X$ which are essentially bounded on $\X$, i.e. such that $\|f\|_{\infty}:=\text{ess sup}_{x\in\X}|f(x)|<\infty$.
Two integers $p$ and $q$ are said to be conjugate if $1/p + 1/q = 1$, and by Riesz's representation theorem (see e.g.~\cite[Theorem~2.14]{lieb2001analysis}), {the dual space of $L^q(\X)$ for $1 \leq q < \infty$ (i.e. the set of continuous linear functionals on $L^q(\X)$) is isometrically isomorphic to $L^p(\X)$.}

For $\mu \in \mathcal{M}(\X)$, if $\mu \ll \lambda$ then there exists a measurable function $h$ on $\X$ such that $d \mu = h \, d \lambda$ and the function $h$ is called the  {density} of $\mu$. If $h \in L^p(\X)$, by a slight abuse of notation, we write $\mu \in L^p(\X)$ and $\|\mu\|_p:=\|h\|_p$.
If in addition $\mu$ is invariant w.r.t.~$f$ then we say that $f$ has an {invariant density} in $L^p(\X)$.

In Section~\ref{sec:acexists}, we state some conditions fulfilled by the moments of an absolutely continuous measure with a density in $L^p(\X)$. In the case of invariant measures, we rely on these conditions to provide an infinite-dimensional LP characterization in Section~\ref{sec:acconic}. In Section~\ref{sec:acsdp}, we show how to approximate the solution of this LP by using a hierarchy of finite-dimensional SDP relaxations. Section~\ref{sec:acinvdens} is dedicated to the approximation of the invariant density. In Section~\ref{sec:piecewise}, we extend the approach to piecewise-polynomial systems.

\subsection{Invariant Densities in Lebesgue Spaces}
\label{sec:acexists}
\newcommand{\balpha}{{\boldsymbol{\alpha}}}
\newcommand{\bgamma}{{\boldsymbol{\gamma}}}


\begin{theorem}
\label{th:Lpdensity}
Let $p$ and $q$ be conjugate with $1 \leq q < \infty$. Consider a sequence $\y \subset \R$. The following statements are equivalent:
\begin{itemize}
\item[(i)]  $\y$ has a representing measure $\mu \in L^p_+(\X)$ with $ \Vert \mu \Vert_p \leq \gamma < \infty$ for some $\gamma \geq 0$;
\item[(ii)] there exists $\gamma \geq 0$ such that for all $g \in \R[x]$ it holds

\begin{equation} 
\label{eq:dominationAC}
|\ell_{\y}(g)|  \leq \gamma \ell_{\z}(|g|^q)^{1/q}.
\end{equation}
and for all $g \in {\mathcal P}(\X)$, it holds $\ell_{\y}(g) \geq 0$.
\end{itemize}
\end{theorem}	
\begin{proof}

$(i) \implies (ii)$. Let $h \in L^p_+(\X)$ be the density of $\mu$. By H\"{o}lder's inequality, it follows that for all $g \in \R[\x]$,
\begin{align*}
|\ell_\y(g)| \leq \int_\X |g d\mu| = \int_\X |g|\:h d\lambda = \Vert g h\Vert_1 
\leq \gamma \Vert g \Vert_q = \gamma  \ell_{\z}(|g|^q)^{1/q} \,,
\end{align*}
which proves \eqref{eq:dominationAC}. Moreover, from the Riesz-Haviland Theorem (see e.g.~\cite[Theorem~3.1]{lasserre2009moments}) one has $\ell_{\y}(g) \geq 0$ for all $g \in {\mathcal P}(\X)$.
		
$(i) \impliedby (ii)$. Since $\ell_{\y}(g)\geq 0$ for all $g \in {\mathcal P}(\X)$, we rely again on  the Riesz-Haviland Theorem to show that $\y$ has a representing measure $\mu \in {\mathcal M}_+(\X)$. It remains to prove that $\mu \in L_+^p(\X)$. 		
For this, we use a modified version of \cite[Theorem~2,~p.106]{yosida1980functional}, originally stated for complex linear spaces. This theorem relies on the Hahn-Banach extension Theorem in complex linear spaces stated in \cite[p.105]{yosida1980functional}. One obtains a similar version for real linear spaces by using the Hahn-Banach extension Theorem in real linear spaces \cite[p.102]{yosida1980functional}. 

By the real version of of~\cite[Theorem~2,~p.106]{yosida1980functional} with the notations $i \leftarrow \alpha$, $X \leftarrow L_+^q(\X)$, $\beta_i \leftarrow g_\alpha$, $\alpha_i \leftarrow y_\alpha$, $x_i \leftarrow x^\alpha$, $f \leftarrow \tilde{\ell}_\y$, there exists a linear functional $\tilde{\ell}_\y$ in $L^q(\X)$ such that $\tilde{\ell}_\y(\x^\alpha) = y_\alpha$ for all $\alpha \in \N^n$, its operator norm $\Vert \cdot \Vert$ being bounded by a nonnegative real constant  $\gamma$. Moreover, the restriction of $\tilde{\ell}_\y$ to $\R[\x]$ is $\ell_\y$ hence $\Vert \ell_\y \Vert \leq \Vert \tilde{\ell}_\y \Vert \leq \gamma$. Since the dual space of $L_+^q(\X)$ is isometrically isomorphic to $L_+^p(\X)$, there exists $\mu \in L_+^p(\X)$ such that $\Vert \tilde{\ell}_\y \Vert = \Vert \mu \Vert_p \leq \gamma$ and satisfying  $ \tilde{\ell}_\y(g) = \ell_\y(g) = \int_\X g d\mu$, for all $g \in \R[\x]$. 
\end{proof}
Theorem~\ref{th:Lpdensity} provides necessary and sufficient conditions satisfied by the moments of an absolutely continuous Borel measure with a density in $L^p_+(\X)$. We now state further characterizations when $p = q = 2$ in Theorem~\ref{th:L2} and  when $p= \infty$ and $q=1$ in Theorem~\ref{th:Linfty}. For a given sequence $\y = (y_\alpha)_\alpha$ and $r \in \N$, the notation $\y^r$ stands for the truncated sequence $(y_\alpha)_{|\alpha| \leq 2 r}$. The notation $\succeq 0$ means positive semidefinite.
\begin{theorem}
\label{th:L2}
Consider a sequence $\y \subset \R$.
The following statements are equivalent:
\begin{itemize}
\item[(i)] $\y$ has a representing measure $\mu \in L^2_+(\X)$ with $\Vert \mu \Vert_2 \leq \gamma < \infty$ for some $\gamma \geq 0$;
\item[(ii)] there exists $\gamma \geq 0$ such that for all $r \in \N$:
\begin{align}
\label{eq:L2sdp}
\begin{pmatrix}
 \M_r(\z) & \y^{r} \\
 (\y^r)^T & \gamma^2
\end{pmatrix}
\succeq 0
\end{align}
and 
\begin{align}
\label{eq:matlocsdp}
\M_{r-r_j}(g_j \y) \succeq 0 \,, \quad j=0,1,\dots,m \,.
\end{align}
\end{itemize}
\end{theorem}
\begin{proof}
$(i) \implies (ii)$. By using the first necessary condition of Theorem~\ref{th:Lpdensity}, there exists $\gamma \geq 0$ such that $\ell_\y(g) \leq \gamma \ell_\z(g^2)^{1/2}$ for all $g \in \R_{r}[\x]$. 
Let $r \in \N$ and choose an arbitrary $g \in \R_{r}[\x]$ with vector of coefficients $\g$. Thus, one has $\ell_\y(g)^2 \leq \gamma^2 \ell_\z(g^2)$, yielding $(\g^T \y)^2 \leq \gamma^2 \int_\X g^2(\x) d \lambda$. Therefore this implies $\g^T \y^r (\y^r)^T \g \leq \gamma^2 \g^T \M_r(\z) \g$. We obtain~\eqref{eq:L2sdp} by using a Schur complement. We prove that~\eqref{eq:matlocsdp} holds in a similar way by using the second necessary condition of Theorem~\ref{th:Lpdensity}.

$(i) \impliedby (ii)$. Since Assumption~\ref{hyp:archimedean} holds, one can apply Putinar's Positivstellensatz~\cite[Theorem~2.2]{HLS09vol} to prove that~\eqref{eq:matlocsdp} implies that $\y$ has a representing measure $\mu \in {\mathcal M}_+(\X)$. As above, we show that~\eqref{eq:L2sdp} implies that $|\ell_\y(g)| \leq \gamma \ell_\z(g^2)^{1/2}$ for all $g \in \R_{r}[\x]$. We conclude the proof by using the sufficient condition of Theorem~\ref{th:Lpdensity}.
\end{proof}
\begin{theorem}
\label{th:Linfty}
Consider a sequence $\y \in \R$. The following statements are equivalent:
\begin{itemize}
\item[(i)] $\y$ has a representing measure $\mu \in L^\infty_+(\X)$ with $\Vert \mu \Vert_\infty \leq \gamma$ for some $\gamma \geq 0$;
\item[(ii)] there exists $\gamma \geq 0$ such that for all $r \in \N$:
\begin{align}
\label{eq:Linftysdp}
 \gamma \M_r(\z) \succeq \M_r(\y)
\end{align}
\begin{align}
\label{eq:matlocsdp2}
\M_{r-r_j}(g_j \y) \succeq 0 \,, \quad j=0,1,\dots,m \,.
\end{align}
\end{itemize}
\end{theorem}
\begin{proof}

$(i) \implies (ii)$. By using the first necessary condition of Theorem~\ref{th:Lpdensity}, there exists a real $\gamma \geq 0$ such that $(i)$ implies that $|\ell_\y(g^2)| \leq \gamma \ell_\z(g^2)$, for all $g \in \R_{r}[\x]$ with vector of coefficients $\g$. This implies that $\int_\X g(\x)^2 d \mu \leq \gamma \int_\X g(\x)^2 d \lambda$. Since $\int_\X g^2 d \mu = \g^T \M_r(\y) \g$ and $\int_\X g^2 d \lambda = \g^T \M_r(\z) \g$, this shows that $\gamma M_r(\z) \succeq M_r(\y)$. The remaining inequalities are proved as in Theorem~\ref{th:L2}.

$(i) \impliedby (ii)$. As in Theorem~\ref{th:L2}, $\y$ has a representing measure $\mu \in {\mathcal M}_+(\X)$. 
Since~\eqref{eq:Linftysdp} holds, we prove as in~\cite[Lemma 2.4]{HLS09vol} that $\mu \leq \gamma \lambda_\X$ which implies that $\mu \in L^\infty(\X)$ with $\Vert \mu \Vert_\infty \leq \gamma$.
\end{proof}
%

%
In Section~\ref{sec:acsdp}, we will restrict to the case where $p=2$ or $p = \infty$ while relying on the characterizations stated in the two previous theorems.

\subsection{Infinite-dimensional Conic Formulation}
\label{sec:acconic}

Let us consider the following infinite-dimensional conic program:
\begin{equation}
\label{eq:conicac}
\begin{aligned}
\rho^\star_\ac := \sup\limits_{\mu} \quad  & \int_\X  \mu  \\		
\text{s.t.} 
\quad & \Lop(\mu) = 0 \,,\\
\quad & \|\mu\|_p \leq 1 \,,\\
\quad & \mu \in  L^p_+(\X) \,.\\
\end{aligned}
\end{equation}
%

\begin{theorem}
\label{th:lpacprelim}
Problem~\eqref{eq:conicac} admits an optimal solution. If the optimal value $\rho^\star_\ac$ is positive, then the optimal solution is a nonzero invariant measure. 
\end{theorem}
\begin{proof}
First, let us prove that the feasible set of problem~\eqref{eq:conicac} is nonempty and weak-star compact. 
Nonemptiness follows from the fact the zero measure is admissible. Now, let us consider an admissible sequence $(\mu_n)_{n\in\N}$. One has $\Vert \mu_n \Vert_\tv  = \Vert \mu_n \Vert_p \leq 1  < \infty$, which shows that the feasible set of problem~\eqref{eq:conicac} is bounded for the weak-star topology. 
Now, let us assume that the sequence converges weakly-star to $\mu$. For all $\A \in \mathcal{B}(\X)$, one has $0 = \Lop(\mu_n)(\A) \to \Lop(\mu)(\A)$ as $n$ tends to infinity. 
In addition, one has $\Vert \mu_n \Vert_p \leq 1$ which yields $\Vert \mu \Vert_p \leq 1$ as $n$ tends to infinity. Thus $\mu$ is feasible for problem~\eqref{eq:conicac}, which proves that the feasible set of problem~\eqref{eq:conicac} is closed in the metric induced by the weak-star topology. 
This proves that this feasible set is weak-star compact. Problem~\eqref{eq:conicac} has a linear cost function and a weak-star compact feasible set, which implies the existence of an optimal solution. The proof of the second statement is straighforward.
\end{proof}
\begin{assumption}
\label{hyp:invac}
There exists a unique invariant probability measure $\mu_{\ac} \in L^p(\X)$ for some $p \geq 1$.
\end{assumption}
{Note that Assumption~\ref{hyp:invac} is equivalent to supposing that there exists a unique ergodic probability measure.}
\begin{theorem}
\label{th:lpac}
If Assumption~\ref{hyp:invac} holds, then problem~\eqref{eq:conicac} admits a unique optimal solution $\mu_{\ac}^\star := \rho^\star_\ac \, \mu_{\ac}$.
\end{theorem}
\begin{proof}
\if{
First, let us prove that the feasible set of problem~\eqref{eq:conicac} is nonempty and weak-star compact. 
Nonemptiness follows from the fact the zero measure is admissible. Now, let us consider an admissible sequence $(\mu_n)_{n\in\N}$. One has $\Vert \mu_n \Vert_\tv  = \Vert \mu_n \Vert_p \leq 1  < \infty$, which shows that the feasible set of problem~\eqref{eq:conicac} is bounded for the weak-star topology. 
Now, let us assume that the sequence converges weakly-star to $\mu$. For all $\A \in \mathcal{B}(\X)$, one has $0 = \Lop(\mu_n)(\A) \to \Lop(\mu)(\A)$ as $n$ tends to infinity. 
In addition, one has $\Vert \mu_n \Vert_p \leq 1$ which yields $\Vert \mu \Vert_p \leq 1$ as $n$ tends to infinity. Thus $\mu$ is feasible for problem~\eqref{eq:conicac}, which proves that the feasible set of problem~\eqref{eq:conicac} is closed in the metric induced by the weak-star topology. 
This proves that this feasible set is weak-star compact. Problem~\eqref{eq:conicac} has a linear cost function and a weak-star compact feasible set, which implies the existence of an optimal solution.
}\fi

If Assumption~\ref{hyp:invac} holds, then the nonzero invariant measure $\mu_{\ac} / \Vert \mu_{\ac} \Vert_p$ is feasible for problem~\eqref{eq:conicac}, which proves that $\rho_\ac^\star \geq \int_\X \mu_{\ac} / \Vert \mu_{\ac} \Vert_p =  1 / \Vert \mu_{\ac} \Vert_p > 0$.
Finally, let $\mu_\ac^\star$ be an optimal solution of problem~\eqref{eq:conicac}, yielding the optimal value $\rho_\ac^\star = \int_{\X} \mu_\ac^\star$. Then, the measure $(\rho_\ac^\star)^{-1} \, \mu_\ac^\star$ is an invariant probability measure, ensured to be unique from Assumption~\ref{hyp:invac}, which concludes the proof.

\end{proof}

The choice of maximizing the mass of the invariant measure in problem~\eqref{eq:conicac} is motivated by the following reasons:
\begin{itemize}
\item If we consider to solve only the feasibility constraints associated to problem~\eqref{eq:conicac}, one could end up with a solution being the zero measure, even under Assumption~\ref{hyp:invac}.
\item Enforcing the feasibility constraints by adding the condition for $\mu$ to be a probability measure (i.e.~$\int_\X \mu = \Vert \mu \Vert_1 = 1$) would not provide any guarantee to obtain a feasible solution as the inequality constraints $\Vert \mu \Vert_p \leq 1$ may not be fulfilled since $\Vert \mu \Vert_1 \leq \vol \X \Vert \mu \Vert_p \leq \Vert \mu \Vert_p$ when $\mu \in L^p(\X)$ for some $p \geq 1$.
\end{itemize}

\subsection{A Hierarchy of SDP Relaxations}
\label{sec:acsdp}

Let
\[
\Cb^2_r(\y):= \left(\begin{matrix}
 \M_r(\z) & \y^r \\
 (\y^r)^T & 1
\end{matrix}\right), \quad
\Cb^{\infty}_r(\y):=\M_r(\z) - \M_r(\y),
\]
and from now on, let $p=2$ or $p=\infty$ and $r \in \N$ be fixed, with $r \geq r_{\min}$. We build the following hierarchy of finite-dimensional semidefinite programming (SDP) relaxations for problem~\eqref{eq:conicac}:
\begin{equation}
\label{eq:sdpac}
\begin{aligned}
\rho^r_\ac := \sup\limits_{\y} \quad & y_{0} \\			
\text{s.t.} 
\quad & \linfun_{\y}(\x^\alpha) = 0 \,,
\quad \forall \alpha \in \N_{2 r}^n \,, \\
\quad &  \Cb_r^p(\y) \succeq 0 \,, \\
\quad & \M_{r - r_j} (g_j\, \y) \succeq 0, \quad j = 0,1,\dots, m \,.
\end{aligned}
\end{equation}

\begin{lemma}
\label{th:ACInvariantMeasure}
\label{th:gapac}
Problem~\eqref{eq:sdpac} has a compact feasible set and an optimal solution $\y^r$.
\end{lemma}

\begin{proof}
First, let us note that the zero sequence is feasible for SDP~\eqref{eq:sdpac}. Each diagonal element of $\M_r(\y)$ has to be nonnegative, thus $y_\alpha \geq 0$ for $\alpha\in\N^n_{2r}$. Let $\tau := \max_{\alpha \in \N_{2 r}^n} |z_\alpha|$.
For all $\alpha \in \N_{2 r}^n$, from the constraints $\Cb_r^p(\y) \succeq 0$, we consider the two following cases.
\begin{itemize}
\item For $p=2$, one has $\g^T \y \y^T \g \leq \g^T \M_r(\z)\g$, for each $g \in \R_r[\x]$ with vector of coefficients $\g$. In particular for $g(\x) = \x^\alpha$, one obtains $y_\alpha^2 \leq z_{2 \alpha} \leq \tau$.
\item For $p=\infty$, each diagonal element of $\M_r(\z) - \M_r(\y)$ has to be nonnegative, yielding $0 \leq y_\alpha \leq z_\alpha \leq \tau$. 
\end{itemize}
Therefore, this implies that  the feasible set of problem~\eqref{eq:sdpac} is nonempty and compact, ensuring the existence of an optimal solution $\y^r$. 
\end{proof}
Let us denote by $\R[\x]'$ the dual set of $\R[\x]$, i.e., the linear functionals acting on $\R[\x]$.
\begin{lemma}
\label{th:sdpac}
Let Assumption~\ref{hyp:invac} hold and let $\mu_\ac^\star$ be the unique optimal solution of problem~\eqref{eq:conicac}.
For every $r \geq r_{\min}$, let $\y^r$ be an arbitrary optimal solution of problem \eqref{eq:sdpac} and by completing with zeros, consider $\y^r$ as an element of $\R[\x]'$. 
Then the sequence $(\y^r)_{r\geq r_{\min}} \subset \R[\x]'$ converges pointwise to $\y^\star \in \R[\x]'$, {that is, for any fixed $\alpha \in \N^n$}:
\begin{equation}
\label{eq:yalphaaclimit}
\lim\limits_{r\to +\infty} y_\alpha^r = y_\alpha^\star \,.
\end{equation}
Moreover, $\y^\star$ has representing measure $\mu_\ac^\star$. In addition, one has:
\begin{equation} 
\label{eq:rhoaclimit}
\lim\limits_{r\to +\infty} \rho^r_\ac = \rho_\ac^\star = \Vert \mu^\star \Vert_1.
\end{equation}
\begin{proof}
As in the proof of Theorem~3.4 in~\cite{Las16Decomp}, we show that $(\y^r)_{r \geq r_{\min}}$ converges to $(\y^\star)$ since there exists a subsequence of integers $(r_k)$ with $r_k \geq r_{\min}$ such that:
\begin{equation}
\label{proof:cvgac}
\lim_{k \to \infty} y_\alpha^{r_k} = y_\alpha^\star\,.
\end{equation}
For an arbitrary integer $r \geq r_{\min}$, it follows from~\eqref{proof:cvgac} that $0 \preceq \M_r(\y^\star)$. As a consequence of~\cite[Proposition~3.5]{lasserre2009moments}, $\y^\star$ has a representing measure $\mu \in {\mathcal M}_+(\X)$.  Using the fact that $\y^{r_k}$ is feasible for SDP~\eqref{eq:sdpac} together with~\eqref{proof:cvgac}, one has $\linfun_{\y^\star}(\x^\alpha) = 0$, for all $\alpha \in \N_{2 r}^n$,  
$\Cb_r^p(\y) \succeq 0$ and $\M_{r - r_j} (g_j\, \y) \succeq 0$, for all $j = 0,1,\dots, m$. Thus this determinate representing measure satisfies $\Lop (\mu) = 0$ and has a density in $L^p(\X)$ by using Theorem~\ref{th:L2} for the case $p=2$ and Theorem~\ref{th:Linfty} for the case $p=\infty$.
This proves that $\mu$ is feasible for problem~\eqref{eq:conicac} and ensures that $\rho_\ac^\star \geq \int_\X \mu$. 

Since problem~\eqref{eq:sdpac} is a relaxation of problem~\eqref{eq:conicac}, one has $\rho_\ac^\star \leq \rho_{\ac}^{r_k}$ for all $k \in \N$.
Hence, one has $\rho_\ac^\star \leq \lim_{k \to \infty} \rho_{\ac}^{r_k} = \lim_{k \to \infty} y_0^{r_k} = \int_\X \mu$.
This shows that $\mu$ is an optimal solution of problem~\eqref{eq:conicac}, which is unique from Theorem~\ref{th:lpac}. The accumulation point of $\y^r$ is unique as it is the  moment sequence of $\mu_{\ac}^\star$, yielding~\eqref{eq:yalphaaclimit} and~\eqref{eq:rhoaclimit}, the desired results.
\end{proof}
\end{lemma}
\begin{remark}
\label{rk:sdp}
Note that without the uniqueness hypothesis made in Assumption~\ref{hyp:invac}, we are not able to guarantee the pointwise convergence of the sequence of optimal solutions $(\y^r)_{r \geq r_{\min}}$ to $\y^\star$.
\end{remark}
{
\begin{remark}
\label{rk:dual}
One could consider the dual of SDP~\eqref{eq:sdpac}, which is an optimization problem over polynomial sums of squares (SOS). One way to prove the non-existence of invariant densities in $L^p(\X)$ for $p \in \{2,\infty\}$ is to use the output of this dual program, yielding SOS certificates of infeasibility.
\end{remark}
}
\subsection{Approximations of Invariant Densities}
\label{sec:acinvdens}

Recall that $p=2$ or $\infty$. Given a solution $\y^r$ of the SDP~\eqref{eq:sdpac} at finite order $r \geq r_{\min}$, let $h^r \in \R_{2r}[\x]$ be the polynomial with vector of coefficients $\mathbf{h}^r$ given by:
\begin{align}
\label{eq:hrp}
\mathbf{h}^r := \M_r(\z)^{-1} \y^r
\end{align}
where the moment matrix $\M_r(\z)$  is positive definite hence inversible for all $r \in \N$. Note that the degree of the extracted invariant density depends on the SDP relaxation order $r$, and higher relaxation orders lead to higher degree approximations.

\begin{lemma}
\label{th:density}
Let Assumption~\ref{hyp:invac} hold.
For every $r \geq r_{\min}$, let $\y^r$ be an optimal solution of SDP~\eqref{eq:sdpac}, let $h^r$ be the corresponding polynomial obtained as in \eqref{eq:hrp} and let $\mu_\ac^\star$ be the unique optimal solution of problem \eqref{eq:conicac} with density $h^\star_\ac$.
Then, the following convergence holds:
$$
\lim_{r\to +\infty} \int_{\X} g(\x) \, h^r(\x)  d\lambda = \int_{\X} g(\x) \, h^\star_\ac(\x)  d\lambda \,,
$$
for all $g \in \R[\x]$.
\end{lemma}

\begin{proof}
By definition~\eqref{eq:hrp}, one has:
$$
\int_{\X} \x^\alpha h^r(\x)d\lambda = \left[\M_r(\z)\h^r \right]_\alpha = y^r_\alpha \,,
$$
for all $\alpha \in \N^n_{2r}$. As a consequence of~\eqref{eq:yalphaaclimit} from Lemma~\ref{th:sdpac}, this yields:
$$
 \int_\X \x^\alpha h^r(\x)d\lambda = y_\alpha^r \to y_\alpha^\star = \int_\X \x^\alpha h^\star_\ac(\x)d\lambda,
$$
as $r$ tends to infinity. Thus, for all $g \in \R[\x]$, $\int_\X g(\x) \, h^r(\x) d\lambda \to \int_\X g(\x) \, h^\star_\ac(\x) d\lambda$, yielding the desired result.
\end{proof}

\subsection{Extension to Piecewise Polynomial Systems}
\label{sec:piecewise}
Now we explain how to extend the current methodology to piecewise polynomial systems. The idea, inspired from~\cite{Switch13}, consists in using the piecewise structure of the dynamics and the state-space partition to decompose the invariant measure into a sum of local invariant measures supported on each partition cell while being invariant w.r.t.~the local dynamics. 

Let us consider a set of cell indices $I$ and a union of semialgebraic cells $\X = \bigcup_{i \in I} \X_i \subset \R^n $ partitioning the state-space. For each $i \in I$, the state-space cell $\X_i$ is assumed to be a compact basic semialgebraic set:
\begin{equation} \label{eq:defXi}
 \X_i :=  \{\x \in \R^n : g_{i,1}(\x)  \geq 0, \dots, g_{i,m_i} (\x) \geq 0 \} \,,
\end{equation} 
defined by given polynomials $g_{i,1},\dots,g_{i,m_i} \in \R[\x]$, $m_i \in \N$ and fulfilling Assumption~\ref{hyp:archimedean} as well as Assumption~\ref{hyp:momb}. We set $r_{i,j} := \lceil \deg g_{i,j} / 2 \rceil$, for all $j=0,1,\dots,m_i$ and $i \in I$.
Then, one considers either the following discrete-time piecewise polynomial system: 
\begin{align}
\label{eq:pwdisc}
\x_{t+1} = f_i(\x_t) \quad \text{for } \x \in \X_i \,, \quad i \in I \,, \quad t \in \N \,,
\end{align}
or the following continuous-time piecewise polynomial system:
\begin{align}
\label{eq:pwcont}
\dot{\x} = f_i(\x) \quad \text{for } \x \in \X_i \,, \quad i \in I \,, \quad t \in [0, \infty) \,.
\end{align}
For $p \geq 1$, this leads to the following infinite-dimensional conic program:
\begin{equation}
\label{eq:pwconic}
\begin{aligned}
\sup\limits_{\mu_i} \quad  & \sum_{i \in I} \int_{\X_i}  \mu_i  \\		
\text{s.t.} 
\quad & \sum_{i \in I} \mathcal{L}_{f_i} (\mu_i) = 0 \,, \\
\quad & \sum_{i \in I} \Vert  \mu_i \Vert_p \leq 1   \,,\\
\quad & \mu_i \in L^p_+(\X_i) \,, \quad i \in I \,.
\end{aligned}
\end{equation}
Given $\mu_i \in  \mathcal{M}_+(\X_i)$ let us define $\mu := \sum_{i \in I} \mu_i \in \mathcal{M}_+(\X)$ as well as the linear mapping $\Lop : \mathcal{M}_+(\X) \to \mathcal{M}_+(\X)$ by $\Lop (\mu) := \sum_{i \in I} \mathcal{L}_{f_i} (\mu_i)$. Since $\Vert \mu \Vert_p^p = \sum_{i \in I} \Vert \mu_i \Vert_p^p$, one can rewrite 
problem~\eqref{eq:pwconic} as problem~\eqref{eq:conicac}.

Next, we associate to problem~\eqref{eq:pwconic} the hierarchy of SDP relaxations indexed by $r\geq \max_{i\in I} \{ \max_{0 \leq j \leq m_i}\{r_{i,j}\} \}$:
\begin{equation}
\label{eq:pwsdp}
\begin{aligned}
\max\limits_{\y_i} \quad & y_{i,0} \\			
\text{s.t.} 
\quad & \sum_{i \in I} \linfun_{\y_i}(\x^\alpha) = 0 \,,
\quad \forall \alpha \in \N_{2 r}^n \,,  \\
\quad & \sum_{i \in I} \Cb^p_r(\y_i) \succeq 0  \,,  \\
\quad &  \M_{r} (g_{i,j}\, \y_i) \succeq 0 \,, \quad j = 0,1,\dots, m_i \,, \quad i \in I \,.
\end{aligned}
\end{equation}
As for Lemma~\ref{th:sdpac} in Section~\ref{sec:acsdp}, one proves that the sequence of optimal values of SDP~\eqref{eq:pwsdp} converges to the optimal value of problem~\eqref{eq:pwconic}. The extraction of  approximate invariant densities can be performed in a way similar to the procedure described in Section~\ref{sec:acinvdens}.
\section{Singular Invariant Measures} 
\label{sec:sing}
In the sequel, we focus on computing the support of singular measures for either discrete-time or continuous-time polynomial systems. Our approach is inspired from the framework presented in~\cite{Las16Decomp}, yielding a numerical scheme to obtain the Lebesgue decomposition of a measure $\mu$ w.r.t.~$\lambda$, for instance when $\lambda$ is the Lebesgue measure. 
By contrast with~\cite{Las16Decomp} where all moments of $\mu$ and $\lambda$ are {\em a priori} known, we only know the moments of the Lebesgue measure $\lambda$ in our case but we impose an additional constraint on $\mu$ to be an invariant probability measure.

\subsection{Infinite-Dimensional LP Formulation}
\label{sec:singlp}
We start by considering the infinite-dimensional linear optimization problem:
\if{
\begin{equation}
\label{eq:lpsingsimple}
\begin{aligned}
\rho_\sing^\star := \sup\limits_{\mu, \nu} \quad  & \int_\X  \nu  \\		
\text{s.t.} 
\quad & \Lop (\mu) =  0 \,, \quad \int_\X \mu = 1 \,,\\
\quad & \nu \leq \mu  \,, \quad \nu \leq \lambda_\X \,,\\
\quad & \mu, \nu \in  \mathcal{M}_+(\X) \,,\\
\end{aligned}
\end{equation}
which can be equivalently written as follows:
}\fi
\begin{equation}
\label{eq:lpsing}
\begin{aligned}
\rho_\sing^\star = \sup\limits_{\mu, \nu, \hat{\nu}, \psi} \quad  & \int_\X  \nu  \\		
\text{s.t.} 
\quad & \int_\X \mu = 1  \,, \quad \Lop (\mu) =  0  \,,\\
\quad & \nu + \psi = \mu  \,, \quad \nu + \hat{\nu} = \lambda_\X \,,\\
\quad & \mu, \nu, \hat{\nu}, \psi \in  \mathcal{M}_+(\X) \,.\\
\end{aligned}
\end{equation}
%

\begin{assumption}
\label{hyp:inv}
There exists a unique invariant probability measure $\mu^\star \in \mathcal{M}_+(\X)$.
\end{assumption}
For a measure $\nu$ with density $h \in L_+^\infty(\X)$, let us denote by $\max\{1,\nu\}$ the measure with density  $x \mapsto \max\{1,h(x)\} \in L_+^\infty(\X)$.
\begin{theorem}
\label{th:lpsing}
{
Let Assumption~\ref{hyp:inv} hold. Then LP~\eqref{eq:lpsing} has a unique optimal solution $(\mu^\star, \nu_1^\star, \lambda_\X - \nu_1^\star, \mu^\star - \nu_1^\star)$, where $(\nu^\star, \mu^\star - \nu^\star)$ is the Lebesgue decomposition of $\mu^\star$ w.r.t.~$\lambda_\X$ and $\nu_1^\star:=\max\{1,\nu^\star\} \in L^\infty_+(\X)$.}
\end{theorem}
\begin{proof}
We first prove that the feasible set of LP~\eqref{eq:lpsing} is nonempty and weak-star compact. Let us denote by $\mu^\star$ the unique invariant probability  measure for $f$.  Then, nonemptiness follows from the fact that $(\mu^\star, 0, \lambda_\X, \mu^\star)$ is feasible for LP~\eqref{eq:lpsing}. Now, let us consider the sequences of measures $((\mu_n)_n, (\nu_n)_n, (\hat{\nu}_n)_n, (\psi_n)_n)$ such that for all $n \in \N$, $(\mu_n, \nu_n, \hat{\nu}_n, \psi_n)$ is feasible for LP~\eqref{eq:lpsing}. One has  $\| \mu_n \|_\tv = 1 = \| \nu_n \|_\tv + \| \psi_n \|_\tv < \infty$ and $(\hat{\nu})_n \leq \vol \X < \infty$ (as $\X$ is bounded). This shows that the feasible set of LP~\eqref{eq:lpsing} is bounded for the weak-star topology. Now, let us assume that the sequences of measures respectively converge weakly-star to $\mu$, $\nu$, $\hat{\nu}$ and $\psi$. For all $A \in \mathcal{B}(\X)$, one has $0 = \Lop(\mu_n)(A) \to \Lop(\mu) (A)$ and we easily see that $(\mu, \nu, \hat{\nu}, \psi)$ is feasible for LP~\eqref{eq:lpsing}, which proves that the feasible set is closed in the metric including the weak-star topology. LP~\eqref{eq:lpsing} has a linear cost function and a weak-star compact feasible set, thus admits an optimal solution.

The proof that $(\mu^\star, \nu_1^\star, \lambda_\X - \nu_1^\star, \mu^\star - \nu_1^\star)$ is the unique optimal solution of LP~\eqref{eq:lpsing} is similar to the one of Theorem~3.1 in~\cite{Las16Decomp} with the notations $f^\star \leftarrow h^\star$, $\gamma \leftarrow 1$.
\end{proof}
Now we explain the rationale behind LP~\eqref{eq:lpsing}. When there is no absolutely continuous invariant probability measure supported on $\X$, then LP~\eqref{eq:lpsing} has an optimal solution $(\mu^\star, 0, \lambda_\X, \mu^\star)$ with $\mu^\star$ being the unique  singular invariant probability measure. In this case, the value of LP~\eqref{eq:lpsing} is $\rho_\sing^\star = 0$. Note that in the general case where Assumption~\ref{hyp:inv} does not hold, there may be several invariant probability measures. In this case, LP~\eqref{eq:lpsing} still admits an optimal solution and the optimal value is the maximal mass of the $\nu$-component among all invariant probability measures. 

{By contrast with problem~\eqref{eq:conicac} from Section~\ref{sec:ac}, we enforce the feasibility constraints by adding the condition for $\mu$ to be a probability measure. The reason is that if we remove this condition, the value $\rho_\sing^\star = 0$ could still be obtained with another optimal solution $(0, 0, \lambda_\X, 0)$, and we could not retrieve the unique invariant probability measure $\mu^\star$.
}
\subsection{A Hierarchy of Semidefinite Programs}
\label{sec:singsdp}
For every $r \geq r_{\min}$, we consider the following optimization problem:
\begin{equation}
\label{eq:sdpsing}
\begin{aligned}
\rho_{\sing}^r := \sup\limits_{\u, \v, \hat{\v}, \y} \quad & v_0 \\
\text{s.t.} 
\quad & u_0 = 1 \,, \quad \linfun_{\u}(\x^\alpha) =  0 \,,
\quad \forall \alpha \in \N_{2 r}^n \,, \\
\quad & v_{\alpha} + y_{\alpha} = u_\alpha \,, \quad  v_{\alpha} + \hat{v}_{\alpha} = z_\alpha, \quad \forall \alpha \in \N_{2 r}^n \,, \\
\quad & 
\M_{r - r_j}(g_j \, \u) \,,  
\M_{r - r_j}(g_j \, \v) \succeq 0 \,,
\quad j = 0,\dots, m \,,\\ 
\quad &
\M_{r - r_j}(g_j \, \hat{\v}) \,,
\M_{r - r_j}(g_j \, \y) \succeq 0 \,, 
\quad j = 0,\dots, m \,.\\
\end{aligned}
\end{equation}
Problem~\eqref{eq:sdpsing} is a finite-dimensional SDP relaxation of LP~\eqref{eq:lpsing}, implying that $\rho_{\sing}^r \geq \rho_\sing^\star$ for every $r \geq r_{\min}$. 

\begin{theorem}
\label{th:gapsing}
Problem~\eqref{eq:sdpsing} has a compact feasible set and an optimal solution $(\u^\star, \v^\star, \hat{\v}^\star, \y^\star)$.
\end{theorem}
\begin{proof}
%
First, let us denote by $\mu^\star$ the invariant probability measure with associated moment sequence $\u$. Since $\z$ is the moment sequence associated with $\lambda_\X$, it follows that SDP~\eqref{eq:sdpsing} has the trivial feasible solution $(\u, 0, \z, \u)$. 

Then, Assumption~\ref{hyp:archimedean} implies that the semidefinite constraint $\M_{r - 1}(g^\X \u) \succeq 0$ holds. Thus, the first diagonal element of  $\M_{r - 1}(g^\X \u)$ is nonnegative,
and since  $\ell_{\u}(1)  = u_0 = 1$, it follows that $\ell_{\u}(x_i^{2 r}) \leq N^r$, $i=1,\ldots,n$. We deduce from~\cite[Lemma 4.3, p. 111]{LasserreNetzer07SOS} that $|u_\alpha|$ is
bounded for all $\alpha \in \N_{2 r}^n$.
In addition, from the moment constraints of SDP~\eqref{eq:sdpsing}, one has for every $i=1,\ldots,n$:
\[
\ell_{\v}(x_i^{2 r}) \,, \ell_{\hat{\v}}(x_i^{2 r})  \leq \int_\X x_i^{2 r} d \lambda_\X \,, 
\quad \ell_{\y}(x_i^{2 r}) \leq \ell_{\u} (x_i^{2 r}) \leq N^r \,.
\]
which similarly yields that $|v_\alpha|$, $|\hat{v}_\alpha|$ and $|y_\alpha|$ are all 
bounded for all $\alpha \in \N_{2 r}^n$.
Therefore, we conclude that the feasible set of SDP~\eqref{eq:sdpsing} is compact and that there exists an optimal solution $(\u^\star, \v^\star, \hat{\v}^\star, \y^\star)$.
\end{proof}
\begin{theorem}
\label{th:cvgsing}
Let Assumption~\ref{hyp:inv} hold.
For every $r \geq r_{\min}$, let $(\u^r, \v^r, \hat{\v}^r, \y^r)$ be an arbitrary optimal solution of SDP~\eqref{eq:sdpsing} and by completing with zeros, consider $\u^r$, $\v^r$, $\hat{\v}^r$, $\y^r$ as elements of $\R[\x]'$. 
The sequence $(\u^r, \v^r, \hat{\v}^r, \y^r)_{r \geq r_{\min}}  \subset (\R[\x]')^4$ converges pointwise to $(\u^\star, \v^\star, \hat{\v}^\star, \y^\star) \subset (\R[\x]')^4$,  {that is, for any fixed $\alpha \in \N^n$}:
\begin{equation}
\label{eq:cvg}
\lim_{r \to \infty} u_\alpha^r = u_\alpha^\star \,, \quad
\lim_{r \to \infty} v_\alpha^r = v_\alpha^\star \,, \quad
\lim_{r \to \infty} \hat{v}_\alpha^r = z_\alpha^\X -  v_\alpha^\star \,, \quad
\lim_{r \to \infty} y_\alpha^r = u_\alpha^\star - v_\alpha^\star \,.
\end{equation}
Moreover, with $(\mu^\star, \nu_1^\star, \lambda_\X - \nu_1^\star, \mu^\star - \nu_1^\star)$ being the unique optimal solution of LP~\eqref{eq:lpsing}, $\u^\star$ is the moment sequence of the unique invariant probability measure $\mu^\star$, $\v^\star$ and $\y^\star$ are the  respective moment sequences of $\nu_1^\star$, $\hat{\nu}^\star = \lambda_\X - \nu_1^\star$, $\mu^\star - \nu_1^\star$.

In addition, one has:
\[
\lim_{r \to \infty} \rho_\sing^r = \rho_\sing^\star\,. 
\]
\end{theorem}
\begin{proof}
As in the proof of Theorem~3.4 in~\cite{Las16Decomp}, we show that $(\u^r, \v^r, \hat{\v}^r, \y^r)_{r \geq r_{\min}}$ converges pointwise to $(\u^\star, \v^\star, \hat{\v}^\star, \y^\star)$ since there exists a subsequence of integers $(r_k)$ with $r_k \geq r_{\min}$ such that:
\begin{equation}
\label{proof:cvg}
\lim_{k \to \infty} u_\alpha^{r_k} = u_\alpha^\star \,, \quad
\lim_{k \to \infty} v_\alpha^{r_k} = v_\alpha^\star \,, \quad
\lim_{k \to \infty} \hat{v}_\alpha^{r_k} = \hat{v}_\alpha^\star \,, \quad
\lim_{k \to \infty} y_\alpha^{r_k} = y_\alpha^\star\,.
\end{equation}
By fixing an arbitrary integer $r \geq r_{\min}$, it follows from~\eqref{proof:cvg} that $0 \preceq \M_r(\u^\star)$, $0 \preceq \M_r(\v^\star)$, $0 \preceq \M_r(\v^\star)$ and $0 \preceq \M_r(\y^\star)$. Therefore, by using~\cite[Proposition~3.5]{lasserre2009moments}, $\u^\star, \v^\star, \hat{\v}^\star$ and $\y^\star$ are the respective moment sequences of determinate representing measures $\mu$, $\nu$, $\hat{\nu}$ and $\psi$, supported on $\X$.  Using the fact that $(\u^{r_k}, \v^{r_k}, \hat{\v}^{r_k}, \y^{r_k})$ is feasible for SDP~\eqref{eq:sdpsing} together with~\eqref{proof:cvg}, these determinate representing measures must satisfy $\Lop (\mu) = 0$, $\int_\X \mu = 1$, $\nu + \psi = \mu$ and $\nu + \hat{\nu} = \lambda_\X$. This proves that $(\mu, \nu, \hat{\nu}, \psi)$ is feasible for LP~\eqref{eq:lpsing}, thus $\rho_\sing^\star \geq \int_\X \nu$. In addition, one has $\rho_\sing^\star \leq \lim_{k \to \infty} \rho_\sing^{r_k} = \lim_{k \to \infty} v_0^{r_k} = \int_\X \nu$.
This shows that $(\mu, \nu, \hat{\nu}, \psi)$ is an optimal solution of LP~\eqref{eq:lpsing}, which is unique from Theorem~\ref{th:lpsing}.  All accumulation points of $(\u^r$, $\v^r$, $\hat{\v}^r$, $\y^r)$ are unique as they are the  respective moment sequences of $\mu^\star$, $\nu_1^\star$, $\hat{\nu}^\star = \lambda_\X - \nu_1^\star$, $\mu^\star - \nu_1^\star$ and yields~\eqref{eq:cvg}, the desired result.
\end{proof}
The meaning of Theorem~\ref{th:cvgsing} is similar to the one of Theorem~3.4 in~\cite{Las16Decomp}. By noting $(\nu^\star, \psi^\star)$ the Lebesgue decomposition of the unique invariant probability measure $\mu^\star$, we have the two following cases:
\begin{enumerate}
\item If $\nu^\star \in L^\infty_+(\X)$  with $\|\nu^\star\|_\infty \leq 1$ then we can obtain all the moment sequences associated to $\nu^\star$ and $\psi^\star$ by computing $\v^r$ and $\y^r$ through solving SDP~\eqref{eq:sdpsing} as $r \to \infty$. In~\cite{Las16Decomp}, the sup-norm must be less than an arbitrary fixed $\gamma > 0$ while in the present study we  select $\gamma = 1$ as we consider an invariant probability measure $\mu^\star$. In particular, when there is no invariant measure which is absolutely continuous w.r.t.~$\lambda$, one has $\nu^\star = \nu_1^\star = 0$, $\psi^\star = \mu^\star$ and we obtain in the limit the moment sequence $\y^\star$ of the singular measure $\mu^\star$.

\item 
{If $\nu^\star \notin L^\infty_+(\X)$ or $\nu^\star \in L^\infty_+(\X)$  with $\|\nu^\star\|_\infty > 1$, then the invariant probability measure $\mu^\star$ is equal to $\nu' + \psi'$, with $\nu' = \max\{1,\nu^\star\} \in L^\infty_+(\X)$ and $\psi' = \mu^\star - \nu'$ is not singular w.r.t.~$\lambda$.}
\end{enumerate}

\subsection{Support Approximations for Singular Invariant Measures}
%

\begin{definition}{\textbf{(Christoffel polynomial)}}
Let $\mu \in {\mathcal M}_+(\X)$ be such that its moments are all finite and that for all $r \in \N$, the moment matrix $\M_r(\u)$ is positive definite. With $\v_r(\x)$ denoting the vector of monomials of degree less or equal than $r$, sorted by graded lexicographic order, the Christoffel polynomial is the function $p_{\mu, r} : \X  \to \R$ such that
\[
\x \mapsto p_{\mu, r} (\x) := \v_r(\x)^T \M_r(\u)^{-1} \v_r(\x).
\]
\end{definition}

The following assumption is similar to~\cite[Assumption~3.6 (b)]{Christoffel17}. It provides the existence of a sequence of thresholds $(\alpha_r)_{r \in \N}$ for the Christoffel function associated to a given measure {$\mu$ in order to approximate the support $\S$ of this measure}.
Here, we do not assume as in~\cite[Assumption~3.6 (a)]{Christoffel17} that the closure of the interior of $\S$ is equal to $\S$.
%


\begin{assumption}
\label{hyp:christoffel}
{Given a measure $\mu \in {\mathcal M}_+(\X)$ with support $\S \subseteq \X$}, $\S$ has nonempty interior and there exist three sequences $(\delta_r)_{r \in \N}$, $(\alpha_r)_{r \in \N}$, $(d_r)_{r \in \N}$ such that: 
{
\begin{itemize}
\item $(\delta_r)_{r \in \N}$ is a decreasing sequence of positive numbers converging to 0. 
\item For every $r \in \N$, $d_r$ is the smallest integer such that:
\begin{equation}
\label{eq:christoffel}
2^{3 - \frac{\delta_r d_r}{\delta_r + \diam \S}} 
d_r^n \biggl( \frac{e}{n} \biggr)^n \exp \biggl( \frac{n^2}{d_r} \biggr)
 \leq \alpha_r \,.
\end{equation}
\item For every $r \in \N$, $\alpha_r$ is defined as follows:
\[
\alpha_r := \frac{\delta_r^n \omega_n}{\vol \S}  \frac{(d_r + 1)(d_r+2)(d_r+3)}{(d_r+n+1)(d_r+n+2)(2 d_r+n+6)}
\,,
\]
where $\diam \S$ denotes the diameter of the set $\S$ and $\omega_n := \frac{2 \pi^{\frac{n+1}{2}}}{\Gamma(\frac{n+1}{2})}$ is the surface of the $n$-dimensional unit sphere in $\R^{n+1}$. 
\end{itemize}
}
\end{assumption}
\begin{remark}
Regarding Assumption~\ref{hyp:christoffel}, as mentioned in~\cite[Remark~3.7]{Christoffel17}, $d_r$ is well defined for all $r \in \N$ and the sequence $(d_r)_{r \in \N}$ is nondecreasing. Since $\diam \S \leq \diam \X \leq 1$ and $\vol \S \leq \vol \X \leq 1$, replacing $\diam \S$ and $\vol \S$ by $1$ in~\eqref{eq:christoffel} yields a result similar to Theorem~\ref{th:christoffel}. For a given sequence $(\delta_r)_{r \in \N}$, one can compute recursively $d_r$ as well as the threshold $\alpha_r$ for the Christoffel polynomial.
\end{remark}
\begin{theorem}
\label{th:christoffel}
Let Assumption~\ref{hyp:inv} hold and let $\S \subseteq \X$ be the support of the invariant probability measure $\mu^\star$. Suppose that there exist sequences $(\delta_r)_{r \in \N}$, $(\alpha_r)_{r \in \N}$ and $(d_r)_{r \in \N}$ such that  $\mu^\star$, $\S$, $(\delta_r)_{r \in \N}$, $(\alpha_r)_{r \in \N}$ and $(d_r)_{r \in \N}$ fulfill Assumption~\ref{hyp:christoffel}. For every $r \in \N$, let:
\begin{align}
\label{eq:Sr}
\S^r := \{ \x \in \X : p_{\mu^\star, d_r} (\x) \leq  \frac{\binom{d_r+n}{n}}{\alpha_r}  \} \,.
\end{align}
Then $\lim\limits_{r \to \infty} \sup_{\x \in \S^r} \dist{(\x, \S)} = 0$.
\end{theorem}
\begin{proof}
The proof is similar to the first part of the proof of Theorem~3.8 in~\cite{Christoffel17}, which relies on Assumption~\ref{hyp:christoffel} and~\cite[Lemma 6.6]{Christoffel17}. We only need to show that the Christoffel polynomial associated to the invariant measure $\mu^\star$ is affine invariant. Given an invertible affine mapping $g : \x \mapsto \A \x + \b$, with $\A \in \R^{n \times n}$ and $\b \in \R^n$, we prove that the pushforward measure $\tilde{\mu}^\star := g_\# \, \mu^\star $ has all its moments finite, that the associated moment matrix is positive definite and that $p_{\mu^\star, d_r} (\x) = p_{\tilde{\mu}^\star, d_r} (\A \x + \b)$ for all $\x \in \X$. Since $\X$ is compact, $\tilde{\mu}^\star$ has all its moments finite.
Since $g$ is invertible, there exists an invertible matrix $\G$ such that $v_r(g(\x)) = \G v_r(\x)$ for all $\x \in \X$. 
Denoting by $\tilde{\mathbf u}$ the sequence of moments associated to $\tilde{\mu}^\star$, one has by definition of the pushforward measure:
\begin{align*}
\M_r(\tilde{\u}) & = \int_\X \v_r(\x) \v_r(\x)^T d \tilde{\mu}^\star(\x) = \int_\X \v_r(g(\x) ) \v_r(g(\x))^T d \mu^\star(\x) \\
& = \G \int_\X \v_r(\x ) \v_r(\x)^T d \mu^\star(\x) \G^T = \G \M_r(\u) \G^T \,.
\end{align*}
This proves that the moment matrix $\M_r(\tilde{\u})$ is positive definite. 

In addition, one has for all $\x \in \X$, $p_{\tilde{\mu}^\star, r} (g(\x)) := \v_r(g(\x))^T \M_r(\tilde{\u})^{-1} \v_r(g(\x)) = \v_r(\x)^T \G^T \G \M_r(\u)^{-1} \G^T \G \v_r(\x) = p_{\mu^\star, r} (\x) $.
\end{proof}
{
\begin{remark}
\label{rk:discrete}
In the case when the invariant measure is discrete singular, its moment matrix may not be invertible and we can not approximate the support of this measure with the level sets of the Christoffel polynomial.
In this case, one way to recover the support of the measure is to rely on the numerical linear algebra algorithm proposed in~\cite{Henrion05} for detecting global
optimality and extracting solutions of moment problems. This algorithm is implemented in the GloptiPoly~\cite{gloptipoly} software and has been already used in previous work~\cite{HenrionFixpoints} to recover finite cycles in the context of discrete-time systems.
\end{remark}
}
\section{Numerical Experiments}
\label{sec:bench}
Here, we present experimental benchmarks that illustrate our method. 

In Section~\ref{sec:benchac}, we compute the optimal solution $\y^r$ of the primal SDP program~\eqref{eq:sdpac} for a given positive integer $r$ and $p=2$ or $\infty$, as well as the approximate polynomial density $h_p^r$ defined in~\eqref{eq:hrp}.

In Section~\ref{sec:benchsing}, we compute the optimal solution $\u^r$ of the primal SDP program~\eqref{eq:sdpsing} for a given positive integer $r$ as well as $\S^r$, the sublevel set  of the Christoffel polynomial defined in~\eqref{eq:Sr}. {In practice, the computation of $\alpha_r$ in~\eqref{eq:Sr} relies on the following iterative procedure: we select $d_r = r$, $\delta_r = 1$ and increment the value of $\delta_r$ until the inequality~\eqref{eq:christoffel} from Assumption~\ref{hyp:christoffel} is satisfied.}

SDPs~\eqref{eq:sdpac} and~\eqref{eq:sdpsing} are both modeled through GloptiPoly~\cite{gloptipoly} via {\sc Yalmip} toolbox~\cite{YALMIP} available within {\sc Matlab} and interfaced with the SDP solver {\sc Mosek}~\cite{mosek}. Performance results are obtained with an Intel Core i7-5600U CPU ($2.60\, $GHz) with 16Gb of RAM running on Debian~8. 
For each problem, we apply a preprocessing step which consists in scaling data (dynamics, general state constraints) so that the constraint sets become unit boxes. {Note that our theoretical framework (including convergence of the SDP relaxations) only works after assuming that there exists a unique invariant probability measure (Assumption~\ref{hyp:invac} and Assumption~\ref{hyp:inv}). Even though this may not hold for some of the considered systems, numerical experiments show that satisfying results can be obtained when approximating invariant densities and supports of singular measures.}

\subsection{Absolutely Continuous Measures}
\label{sec:benchac}
\subsubsection{Square Integrable Invariant density}
\label{ex:cfd}
First, let us consider the one-dimensional discrete-time polynomial system defined by
\begin{align*}
t^+ & = T(t):= t + w \mod 1 \,,
\end{align*}
with $t$ {being constrained in} the interval $\T := [0, 1]$ and $w \in \R \backslash \mathbb{Q}$ be an arbitrary irrational number. 
{This dynamics corresponds to the circle rotation with an irrational angle $w$ and thus it has a unique invariant measure equal to the restriction of the Lebesgue measure on $\T$~\cite{Weyl1910}, i.e. Assumption~\ref{hyp:invac} is fulfilled here.}

Let us consider the square integrable probability density $h^\star (t) := \frac{3}{4} t^{-1/4}$ and let $F(t) := \int_0^t h^\star(s) d s = t^{3/4}$ be its cumulative distribution function. Since $F$ is invertible, then  $F^{-1}(t) = t^{4/3}$ is distributed according to $h^\star$ and hence the following dynamical system
\begin{align}
\label{eq:cfd}
x^+ & = F^{-1} \circ T \circ F (x)  \,,
\end{align}
{with $x$ being constrained in} the interval $\X := [0, 1]$, has the invariant measure with density $h^\star \in L^2(\X)$. 
%
\begin{figure}[!ht]
\centering
\subfigure[$r=4$]{
\includegraphics[scale=\sizesmallfig]{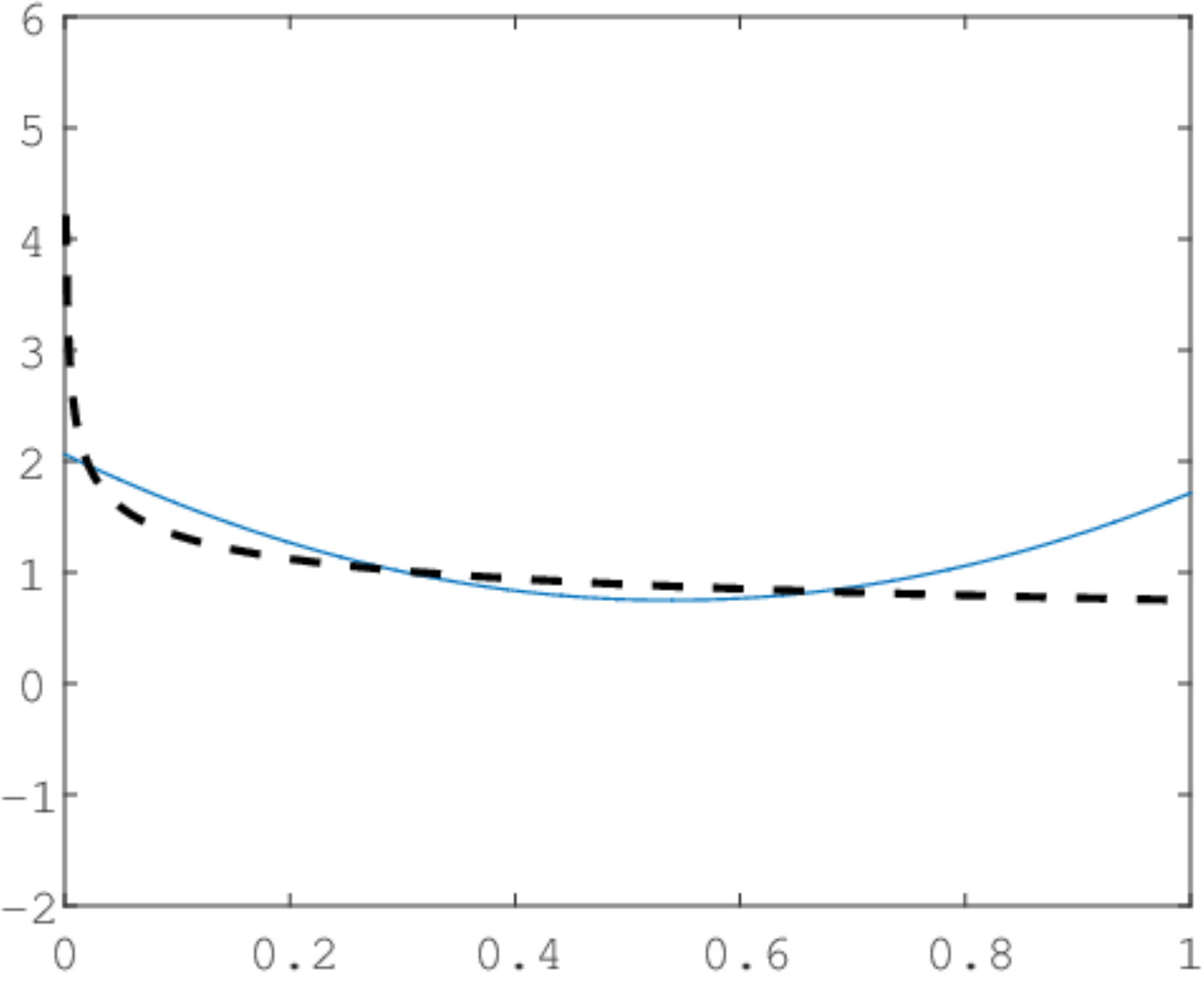}}
\subfigure[$r=6$]{
\includegraphics[scale=\sizesmallfig]{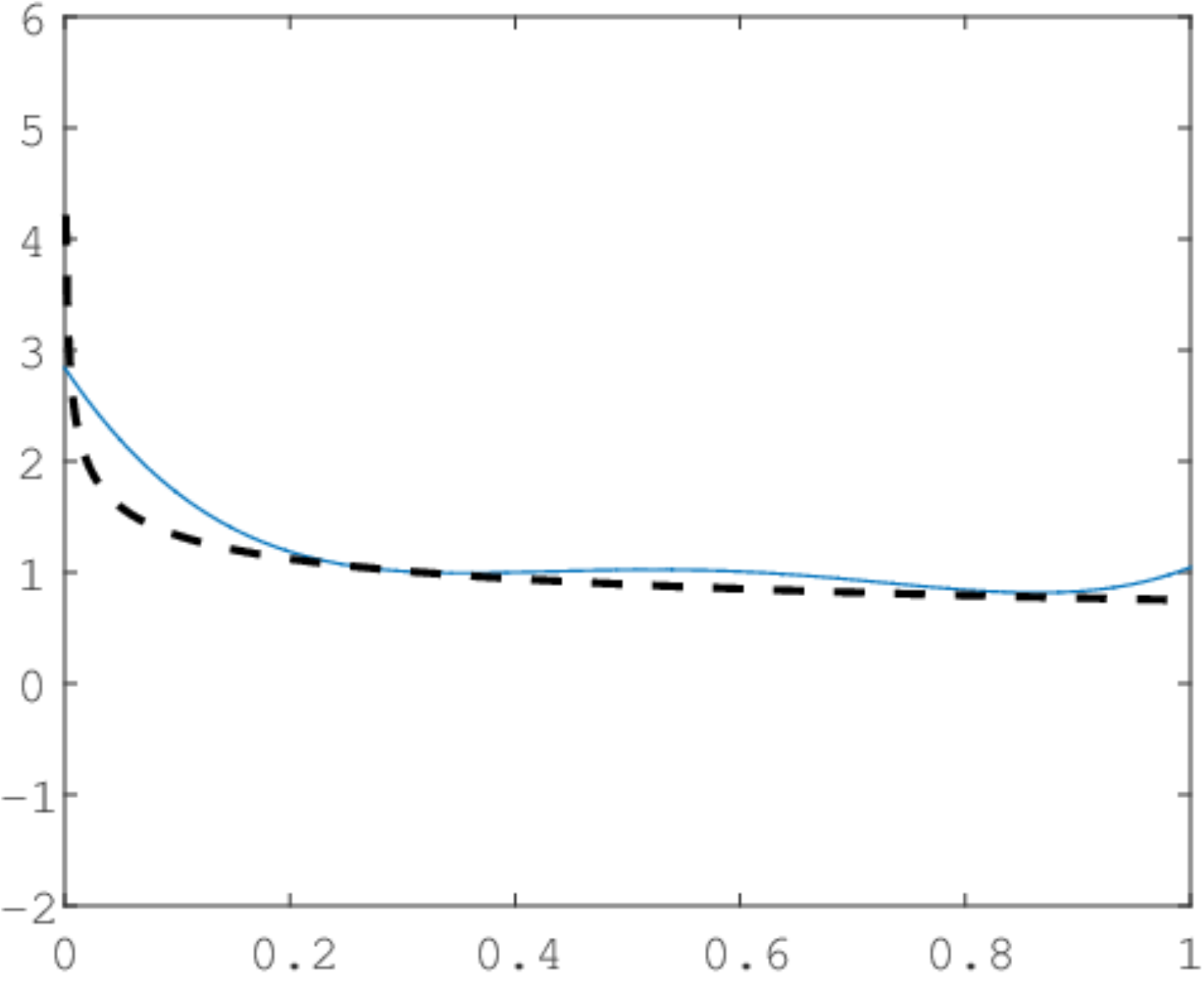}}
\subfigure[$r=8$]{
\includegraphics[scale=\sizesmallfig]{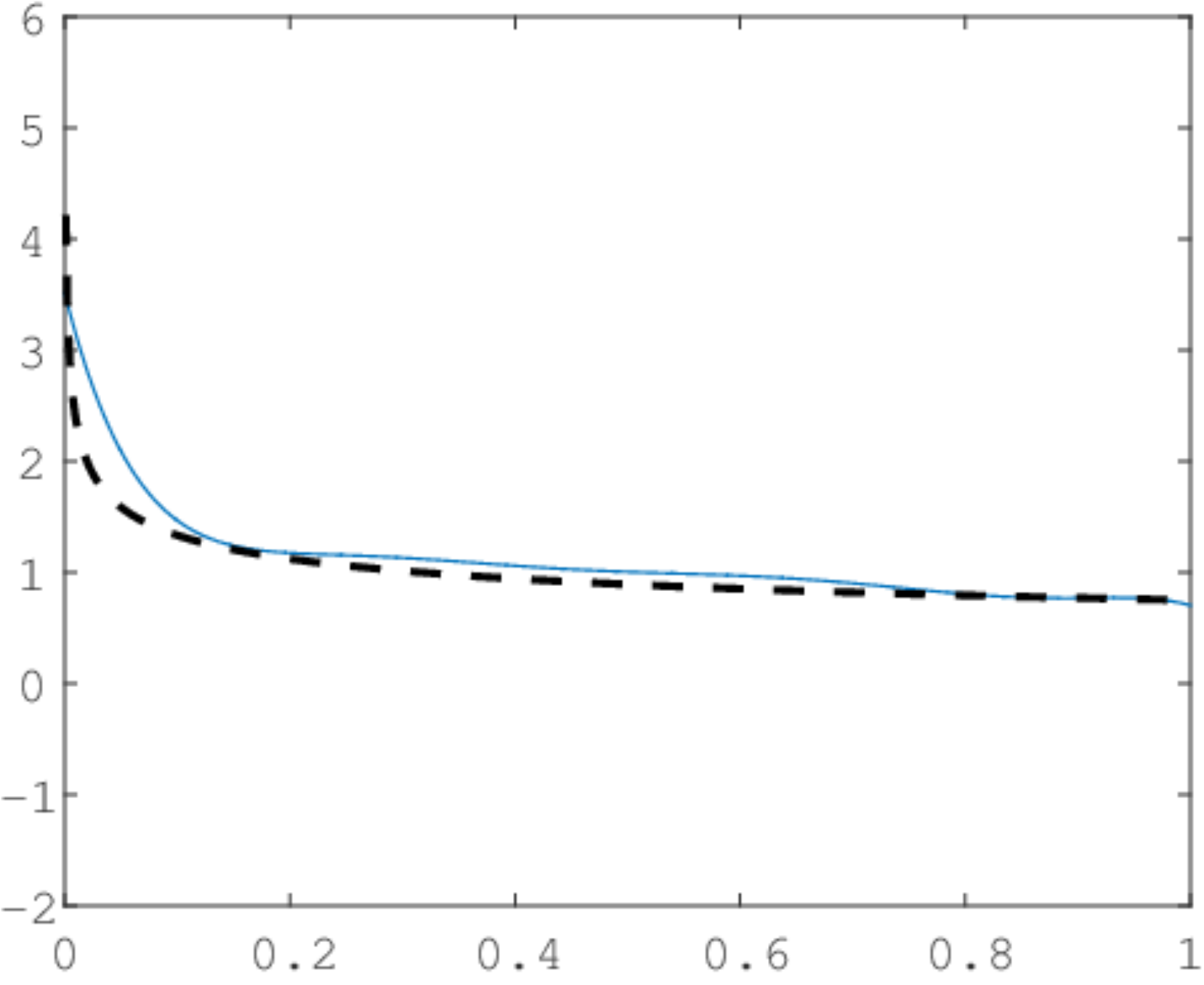}}
\caption{Approximate invariant density for the dynamics from~\eqref{eq:cfd} with corresponding approximations $h_2^r$ (solid curve) of the exact density $h^\star$ (dashed curve) for $r \in \{4,6,8\}$ and $w = \frac{\sqrt{99}}{10}$.}
\label{fig:cfd}
\end{figure}
\if{
Ulam and von Neumann proved in~\cite{Ulam47} that the function $h^\star(x) = \pi^{-1} (x-x^2)^{-1/2}$ is an invariant density for the logistic map. 
Note that the function $h^\star \in L^p(\X)$ for all $p \in [1, 2)$ and so $h^\star \notin L^2(\X)$. However, we could still obtain solutions of SDP~\eqref{eq:sdpac} for $p = 2$, together with the density approximations $h_2^r$, for $r \in \{4,6,8\}$. We conjecture that this apparently strange behavior is due to the fact that Mosek (amongst other SDP solvers) is implemented in double-precision floating-point, thus allowing to provide a solution not violating the precision tolerance of the solver.
}\fi
%

Now, let us take $w \in (0, 1)$.
In this case one has $F^{-1} \circ T \circ F (x) = (x^{3/4} + w)^{4/3}$ if $x^{3/4} + w \leq 1$ and  $F^{-1} \circ T \circ F (x) = (x^{3/4} + w - 1)^{4/3}$ otherwise. 
In order to cast the dynamical system from~\eqref{eq:cfd} as a piecewise polynomial system, we introduce two additional (so-called {\em lifting}) variables $y, z$ and consider the system defined as follows:
\begin{align}
\label{eq:cfdpw}
x^+ = y \quad \text{for } (x,y,z) \in \X_1 \cup \X_2
\end{align}
where $\X_1$ and $\X_2$ are defined by
\begin{align*} 
\label{eq:cfdXi}
\X_1 & :=  \{(x,y,z) \in \R^3 : z (1 - w - z) \geq 0, z^4 = x^3, (z + w)^4 = y^3 \} \,, \\
\X_2 & :=  \{(x,y,z) \in \R^3 : (1 - z) (z + w - 1) \geq 0, z^4 = x^3, (z + w - 1)^4 = y^3 \} \,.
\end{align*} 
Note that the collection $\{\X_1,\X_2\}$ is a partition of $[0, 1]^3$. The variable $z$ represents $x^{3/4}$, for all $x \in [0, 1]$.
The variable $y$ represents either $(x^{3/4} + w)^{4/3}$ on $\X_1$ or $(x^{3/4} + w - 1)^{4/3}$ on $\X_2$. Using the results from Section~\ref{sec:piecewise} with $I = \{1,2\}$, we performed numerical experiments with the irrational number $w = \frac{\sqrt{99}}{10}$.
The approximate density $h_2^r$ {obtained in~\eqref{eq:hrp}} from the $r$ first moments (for $r=4,6,8$) and the exact density $h^\star$ are displayed on Figure~\ref{fig:cfd}. These numerical results indicate that the density approximations become tighter when the value of $r$ increases. 
%
%
\subsubsection{Piecewise Systems}
\label{ex:pw}
Next, we consider three discrete-time piecewise systems coming respectively from Example~3, Example~4 and Example~5 in~\cite{Koda82}. {These three systems are known to have unique invariant densities~\cite[Section 4]{Koda82}, thus Assumption~\ref{hyp:invac} is fulfilled here.}

\begin{equation}
\label{eq:koda3}
x^+ := 
\begin{cases}
\frac{2 x}{1 - x^2}  & \text{if} \ x \in \X_1 := [0, \sqrt{2}-1] \,,\\
\frac{1 - x^2}{2 x}   &\text{if} \ x \in \X_2:=[\sqrt{2}-1, 1] \,,
\end{cases}
\end{equation}
\begin{equation}
\label{eq:koda4}
x^+ := 
\begin{cases}
\frac{2 x}{1 - x}  & \text{if} \ x \in \X_1 := [0, \frac{1}{3}] \,,\\
\frac{1 - x}{2 x}   &\text{if} \ x \in \X_2:=[\frac{1}{3}, 1] \,,
\end{cases}
\end{equation}
\begin{equation}
\label{eq:koda5}
x^+ := 
\begin{cases}
(\frac{1}{8} + 2 (x-\frac{1}{2})^3)^{1/3} +\frac{1}{2}  & \text{if} \ x \in \X_1 := [0, \frac{1}{2}] \,,\\
(\frac{1}{8} - 2 (x-\frac{1}{2})^3)^{1/3} +\frac{1}{2}    &\text{if} \ x \in \X_2:=[\frac{1}{2}, 1] \,.
\end{cases}
\end{equation}
As in Section~\ref{ex:cfd}, we introduce an additional lifting variable to handle either the division or the cube root operator. We compute the approximate density $h_\infty^r$ from~\eqref{eq:hrp} by solving SDP~\eqref{eq:pwsdp}.
For the system from~\eqref{eq:koda3}, SDP~\eqref{eq:pwsdp} yields at $r = 6$ (to four significant digits): 
\[ 
y_0 = 1\,,\ y_1 = 0.3924 \,,\ y_2 = 0.2463 \,,\  y_3 = 0.1827\,,\ y_4 = 0.1476 \,,\ y_5 = 0.1254 \,,\ y_6 = 0.1101 
\,.\] 
These approximate values are close to the moments associated to the exact invariant density $h^\star = \frac{4}{\pi} \cdot \frac{1}{1+x^2}$: 
\[
y_0^\star = 1\,,\ y_1^\star = 0.4412 \,,\ y_2^\star = 0.2732 \,,\  y_3^\star = 0.1954 \,,\ y_4^\star = 0.1512 \,,\ y_5^\star = 0.1230 \,,\ y_6^\star = 0.1035
\,.\]
The approximate density $h_\infty^6$ and the exact density $h^\star$ are displayed on Figure~\ref{fig:pw}(a). Similar results are displayed on Figure~\ref{fig:pw}(b) for the system from~\eqref{eq:koda4} and SDP~\eqref{eq:pwsdp} yields:
\[ 
y_0 = 1\,,\ y_1 = 0.3628 \,,\ y_2 = 0.2181 \,,\  y_3 = 0.1582 \,,\ y_4 = 0.1262 \,,\ y_5 = 0.1063 \,,\ y_6 = 0.0929 
\,,\]
with the moments of the exact density $h^\star(x) = \frac{2}{(1+x)^2}$ being:
\[
y_0^\star = 1\,,\ y_1^\star = 0.3863 \,,\ y_2^\star = 0.2274 \,,\  y_3^\star = 0.1589 \,,\ y_4^\star = 0.1215 \,,\ y_5^\star = 0.0981 \,,\ y_6^\star = 0.0822
\,.\]
For the system from~\eqref{eq:koda5}, the second order SDP relaxation already provides a very accurate approximation of the exact density $h^\star(x) = 12 (x-\frac{1}{2})^2)$, as shown in Figure~\ref{fig:pw}(c).


\begin{figure}[!ht]
\centering
\subfigure[$r=6$, system from~\eqref{eq:koda3}]{
\includegraphics[scale=\sizesmallfig]{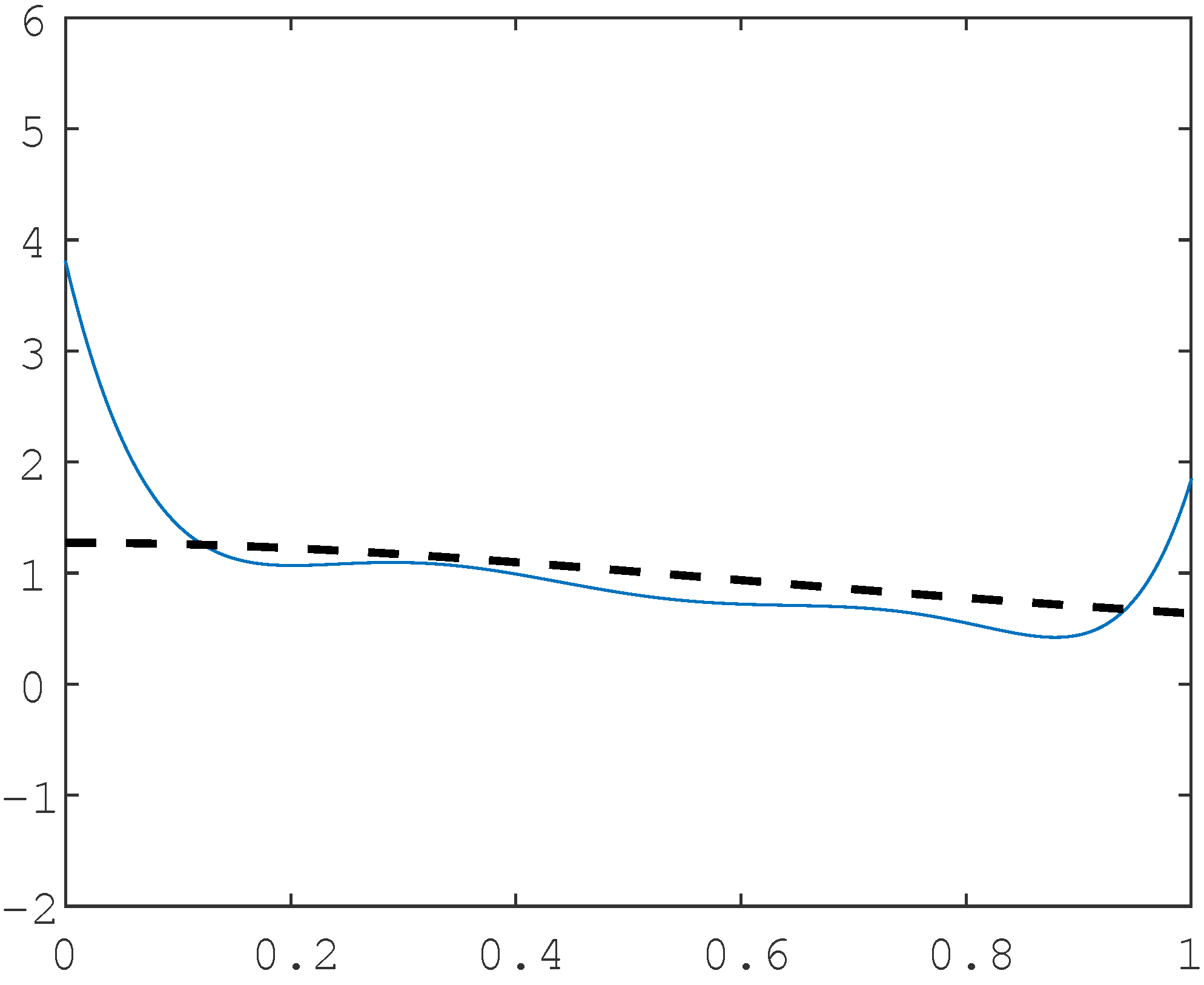}}
\subfigure[$r=6$, system from~\eqref{eq:koda4}]{
\includegraphics[scale=\sizesmallfig]{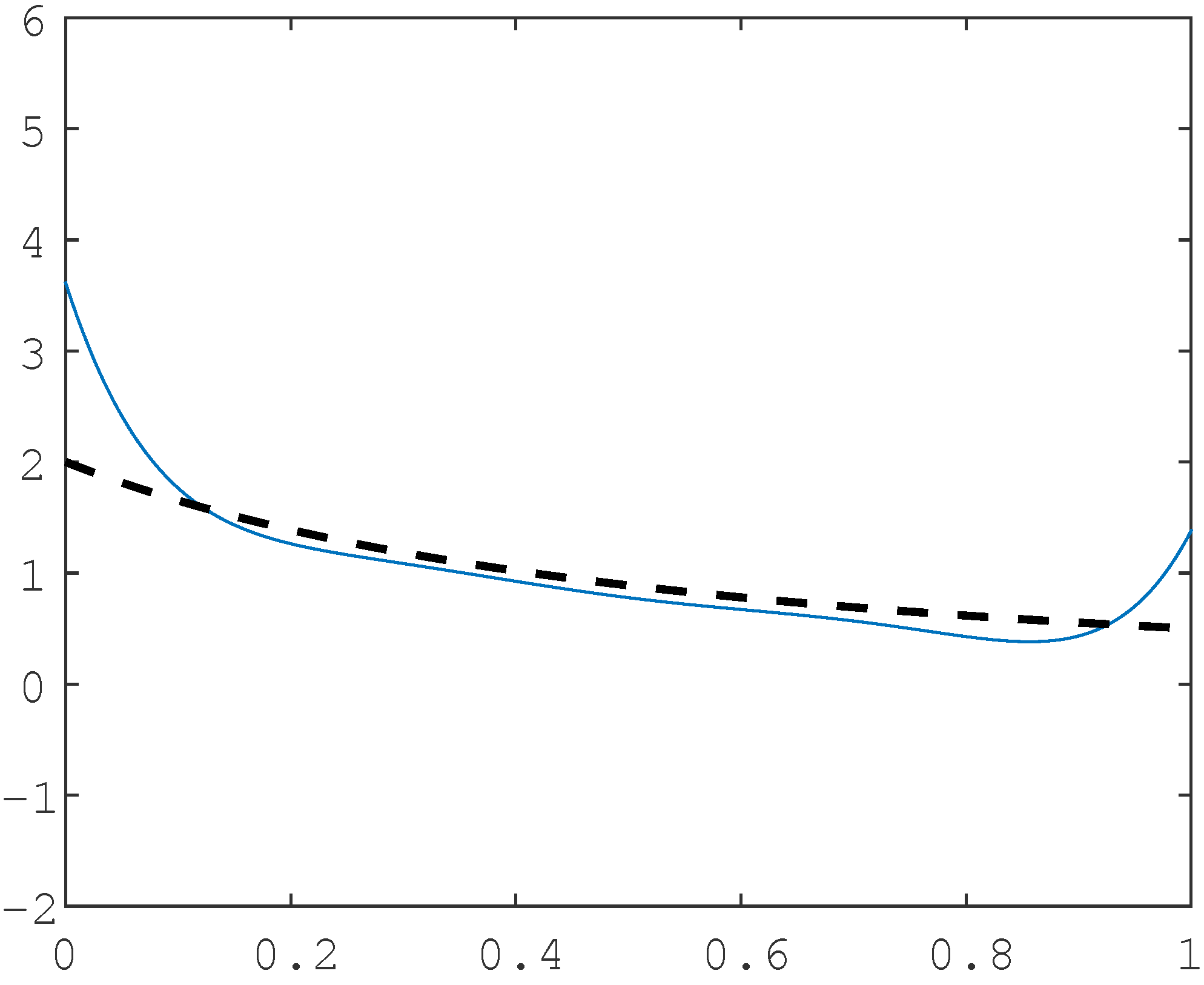}}
\subfigure[$r=2$, system from~\eqref{eq:koda5}]{
\includegraphics[scale=\sizesmallfig]{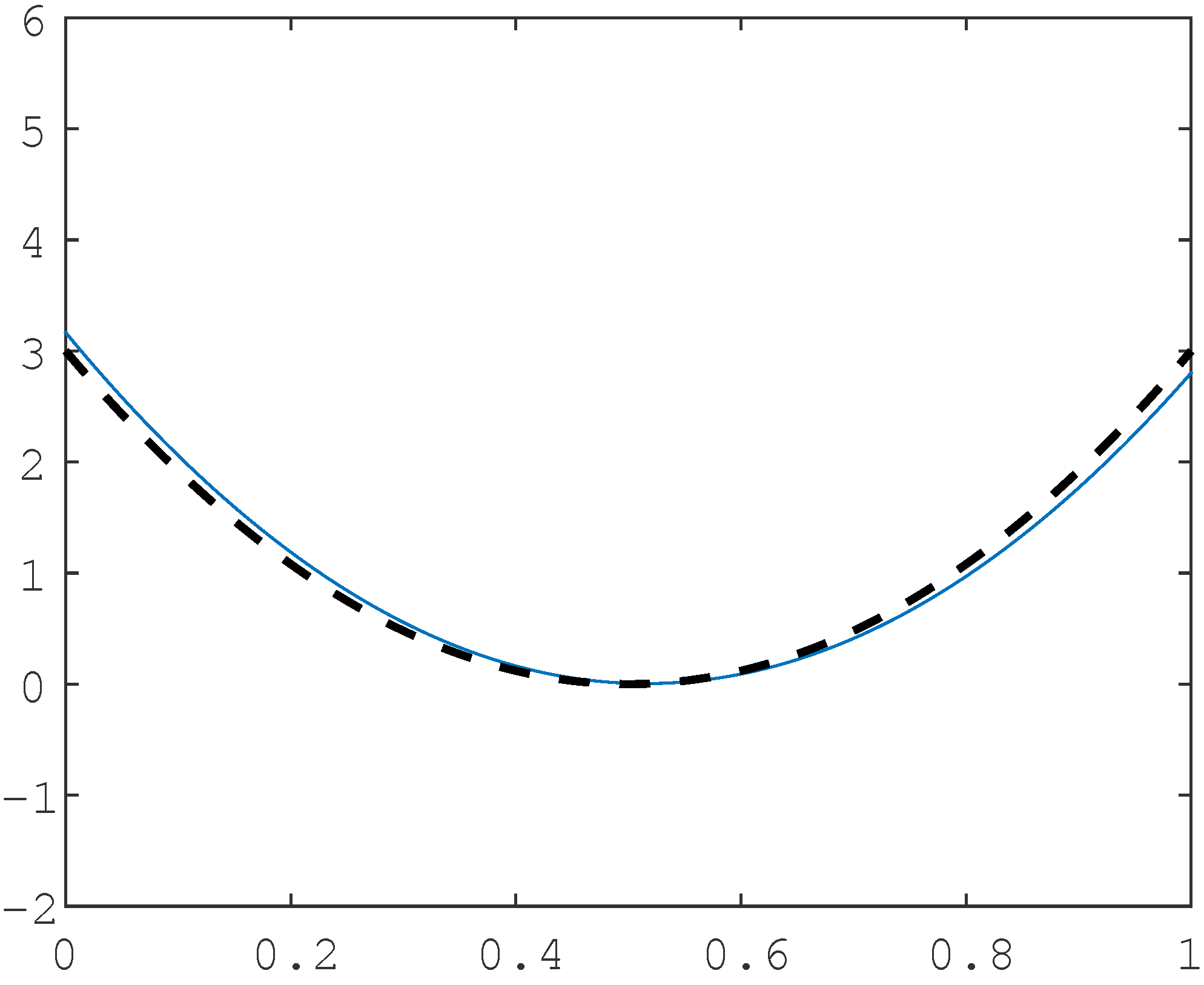}}
\caption{Approximate invariant density for the piecewise systems defined respectively in~\eqref{eq:koda3},~\eqref{eq:koda4} and~\eqref{eq:koda5} with corresponding approximations $h_\infty^r$ (solid curve) of the exact density $h^\star$ (dashed curve).}
\label{fig:pw}
\end{figure}

\if{
Next, we consider the one-dimensional discrete-time piecewise affine system defined by
\begin{align*}
x^+ := 
\begin{cases}
2 x & \text{if} \ x \in \X_1 \,,\\
1 + \frac{3}{2} (\frac{1}{2} - x) &\text{if} \ x \in \X_2 \,.
\end{cases}
\end{align*}
with general state constraints within the intervals $\X_1 := [0, \frac{1}{2}]$ and $\X_2 := [\frac{1}{2}, 1]$.

Using the results from Section~\ref{sec:piecewise} with $I = \{1,2\}$, $f_1(x) = 2 x$ and $f_2(x) = 1 + \frac{3}{2} (\frac{1}{2} - x)$, we compute the approximate density $h_p^r$ {from~\eqref{eq:hrp}} for $r=3,5,7$ and $p = \infty$ by solving SDP~\eqref{eq:pwsdp}. The approximate density $h_p^r$ and the exact density $h^\star = \mathbf{1}_{[\frac{1}{4}, \frac{1}{2}]} + \frac{3}{2} \mathbf{1}_{[\frac{1}{2}, 1]}$ (see~\cite[Example~2]{Koda82} for more details) are displayed on Figure~\ref{fig:pw}.
\begin{figure}[!ht]
\centering
\subfigure[$r=3$]{
\includegraphics[scale=\sizesmallfig]{koda2r3.pdf}}
\subfigure[$r=5$]{
\includegraphics[scale=\sizesmallfig]{koda2r5.pdf}}
\subfigure[$r=7$]{
\includegraphics[scale=\sizesmallfig]{koda2r7.pdf}}
\caption{Approximate invariant density for the piecewise affine map from Example~\ref{ex:pw} with corresponding approximations $h_p^r$ (solid curve) of the exact density $h^\star$ (dashed curve) for $r \in \{3,5,7\}$ and $p = \infty$.}
\label{fig:pw}
\end{figure}

Numerical experiments indicate that the $L^2$-norm of $h_p^r$ is getting closer to the one of $h^\star$ for higher relaxation orders, {consistently with Lemma~\ref{th:density2}}.
However, as emphasized by the plots in Figure~\ref{fig:pw}, it is eventually difficult to approximate the piecewise step function $h^\star \in L^\infty(\X)$ with polynomials of increasing degrees in order to obtain a pointwise convergence behavior. {Such non-uniform convergence behaviors are related to an effect known as the Gibbs phenomenon. We refer to~\cite[Section~5.1]{HDM14meansquared} for a discussion in a similar context.}
To overcome these numerical issues, one way would be to decompose $\X$ into specific subdomains.
}\fi
%
\subsubsection{Rotational Flow Map}
\label{ex:rot}
We consider the rotational flow system, i.e.~the two-dimensional continuous-time  system defined by
\begin{align*}
\dot{x_1} & =  \ \ x_2 \,, \\
\dot{x_2} & = -  x_1  \,.
\end{align*}
with general state constraints within the unit disk $\X := \{ \x \in \R^2 : \|\x\|_2 \leq 1 \}$. 
The restriction of the Lebesgue measure to $\X$ is invariant for the rotational flow map $f(\x) := (x_2,-x_1)$. Indeed, denoting by $\text{n}(\x)$ the outer normal vector to the unit circle $\partial \X$ at $\x$, the Green-Ostrogradski formula yields the following, for all $v \in \mathcal{C}^1(\X)$:
\[
\int_\X \mathrm{grad}\: v \cdot f(\x) d \x = 
\int_{\partial \X} v(\x) \bigl(\underbrace{ f(\x) \cdot d \text{n}(\x) }_{ = 0} \bigr) 
-
\int_\X v(\x) \underbrace{\mathrm{div}\: f(\x)}_{=0} d \x = 0 \,,
\]
{which shows the invariance of $\lambda_\X$.}

For $p = \infty$, SDP~\eqref{eq:sdpac} yields at $r = 2$ (to four significant digits): $y_0 = 1$, $y_1 = y_2 = 0$ and we obtain the approximate polynomial density $h_\infty^r(\x) = 1$, matching the exact invariant density $h^\star$. 

%
\subsection{Singular Measures}
\label{sec:benchsing}
\subsubsection{H\'enon Map}
\label{ex:henon}
The H\'enon map is a famous example of two-dimensional discrete-time systems that exhibit a chaotic behavior.
The system is defined as follows:
\begin{align*}
x_1^+ & =  1 - a x_1^2 + x_2 \,, \\
x_2^+ & = b x_1  \,.
\end{align*}
with general state constraints within the box $\X := [-3, 1.5] \times [-0.6, 0.4]$. 
\begin{figure}[!ht]
\centering
\subfigure[$r=4$]{
\includegraphics[scale=\sizesmallfig]{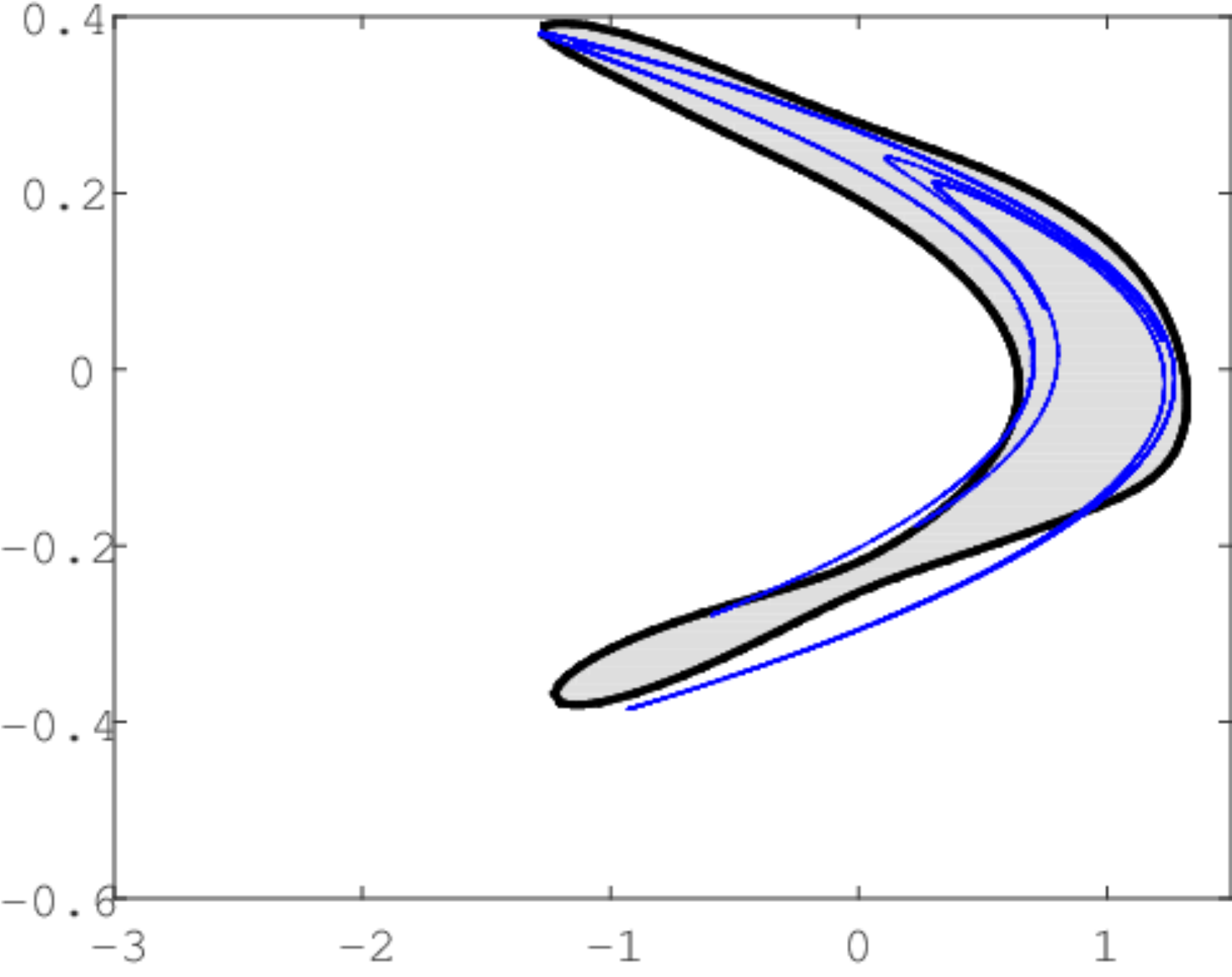}}
\subfigure[$r=6$]{
\includegraphics[scale=\sizesmallfig]{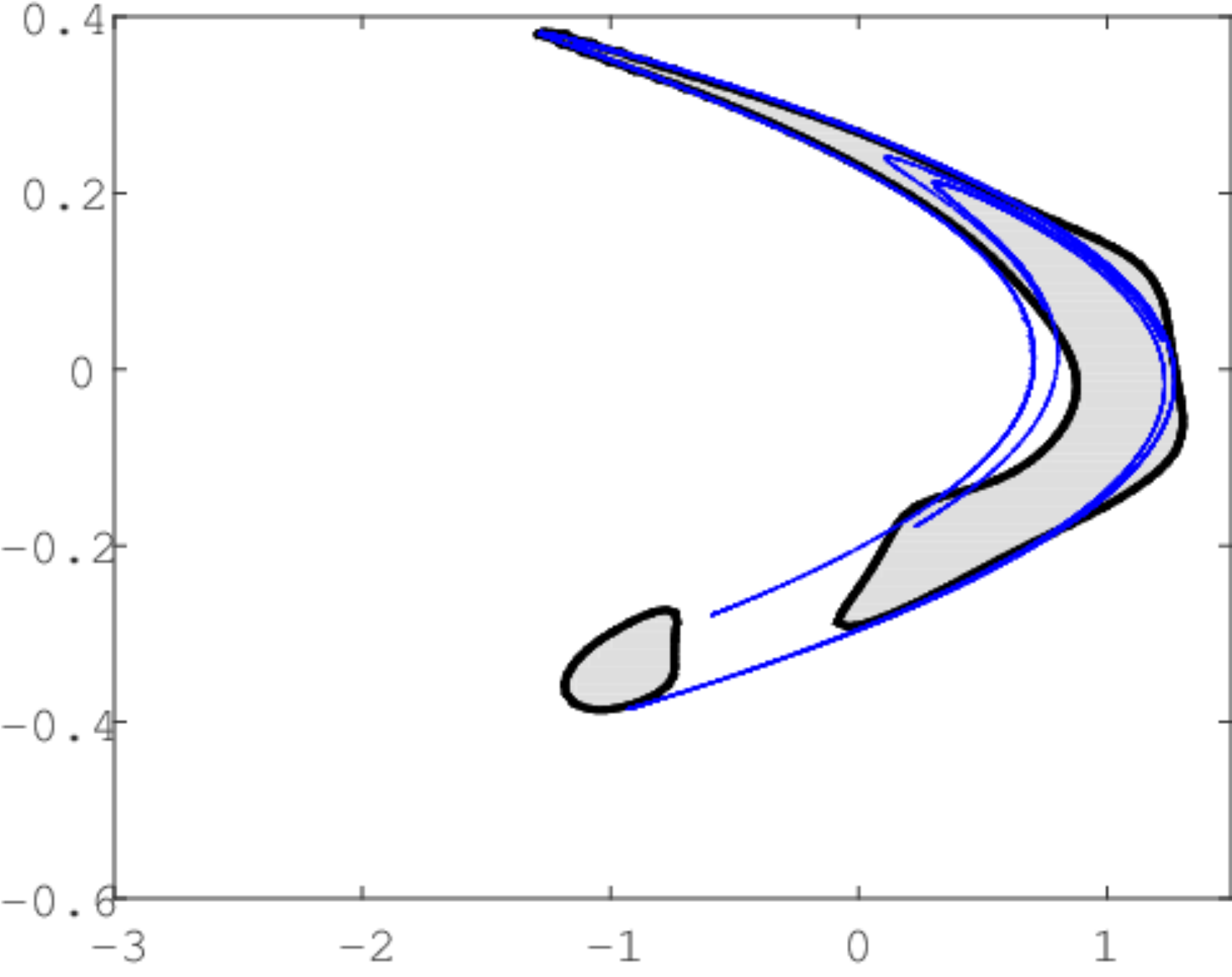}}
\subfigure[$r=8$]{
\includegraphics[scale=\sizesmallfig]{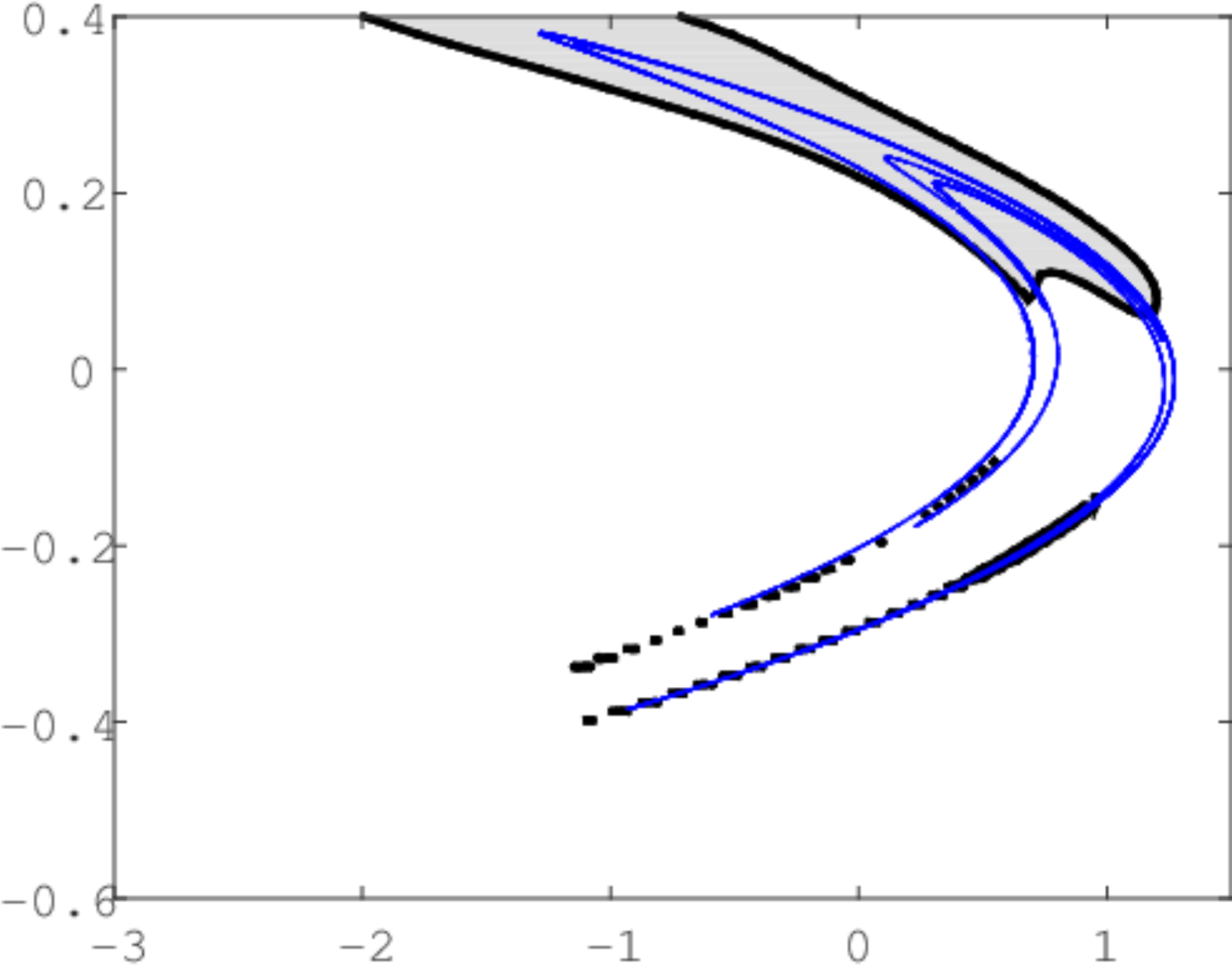}}
\caption{H\'enon attractor (blue) and approximations $\S^r$ (light gray) for the support of the invariant measure w.r.t.~the map from Example~\ref{ex:henon} for $r \in \{4,6,8\}$, $a = 1.4$ and $b = 0.3$.}
\label{fig:henon}
\end{figure}
For $a = 1.4$ and $b = 0.3$, H\'enon proves in~\cite{Henon76} that the sequence of points obtained by iteration of this map from an initial point can either diverge or tend to a strange attractor being the product of a one-dimensional manifold by a Cantor set. 

Figure~\ref{fig:henon} displays set approximations $\S^r$ {obtained from~\eqref{eq:Sr}} of the support of the measure invariant w.r.t.~the H\'enon map, for  $r \in \{4,6,8\}$. For comparison purpose, we also represented the ``true'' H\'enon attractor by displaying the sequence of points obtained after hundred iterations of the map while starting from random sampled initial conditions within the disk of radius $0.1$ and center $[-1, 0.4]$. These numerical experiments show that the level sets of the Christoffel polynomial provide fairly tight approximations of the attractor for modest values of the relaxation order.

\subsubsection{Van der Pol Oscillator}
\label{ex:vanderpol}
The Van der Pol oscillator is an example of an oscillating system with nonlinear damping~\cite{VanderPol26}. The dynamics are given by the following second-order ordinary differential equation:
\begin{align*}
\ddot{x_1} - a (1 - x_1^2) \dot{x_1} + x_1 & =  0 \,.
\end{align*}
By setting $x_2 = \dot{x_1}$, one can reformulate this one-dimensional system into a two-dimensional continuous-time  system:
\begin{align*}
\dot{x_1} & =  x_2 \,, \\
\dot{x_2} & = a (1 -x_1^2) x_2 - x_1  \,.
\end{align*}
When $a > 0$, there exists a limit cycle for the system.
Here, we consider $a = 0.5$ and general state constraints within the box $\X := [-3, 3] \times [-4, 4]$. 
\begin{figure}[!ht]
\centering
\subfigure[$r=4$]{
\includegraphics[scale=\sizesmallfig]{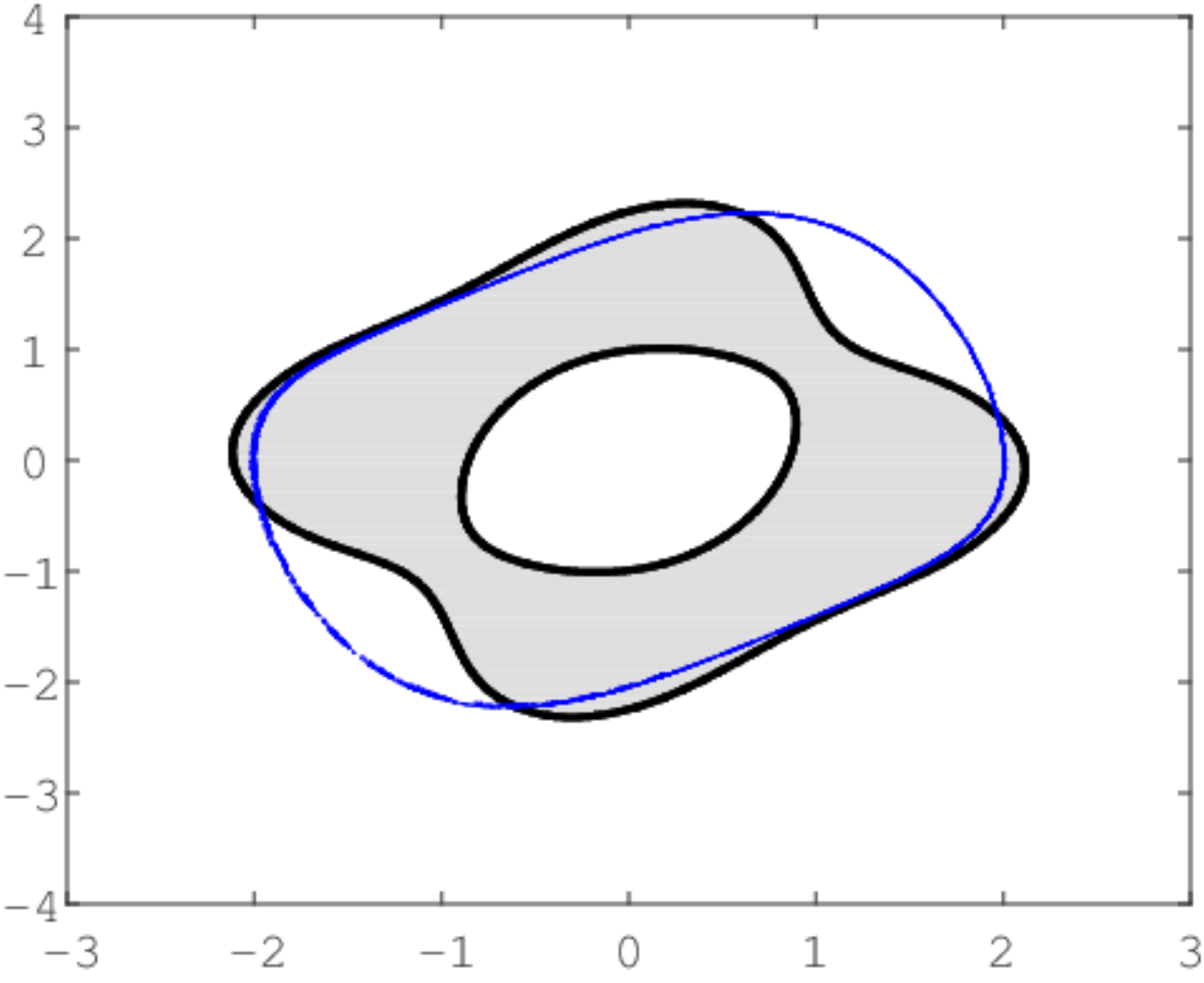}}
\subfigure[$r=6$]{
\includegraphics[scale=\sizesmallfig]{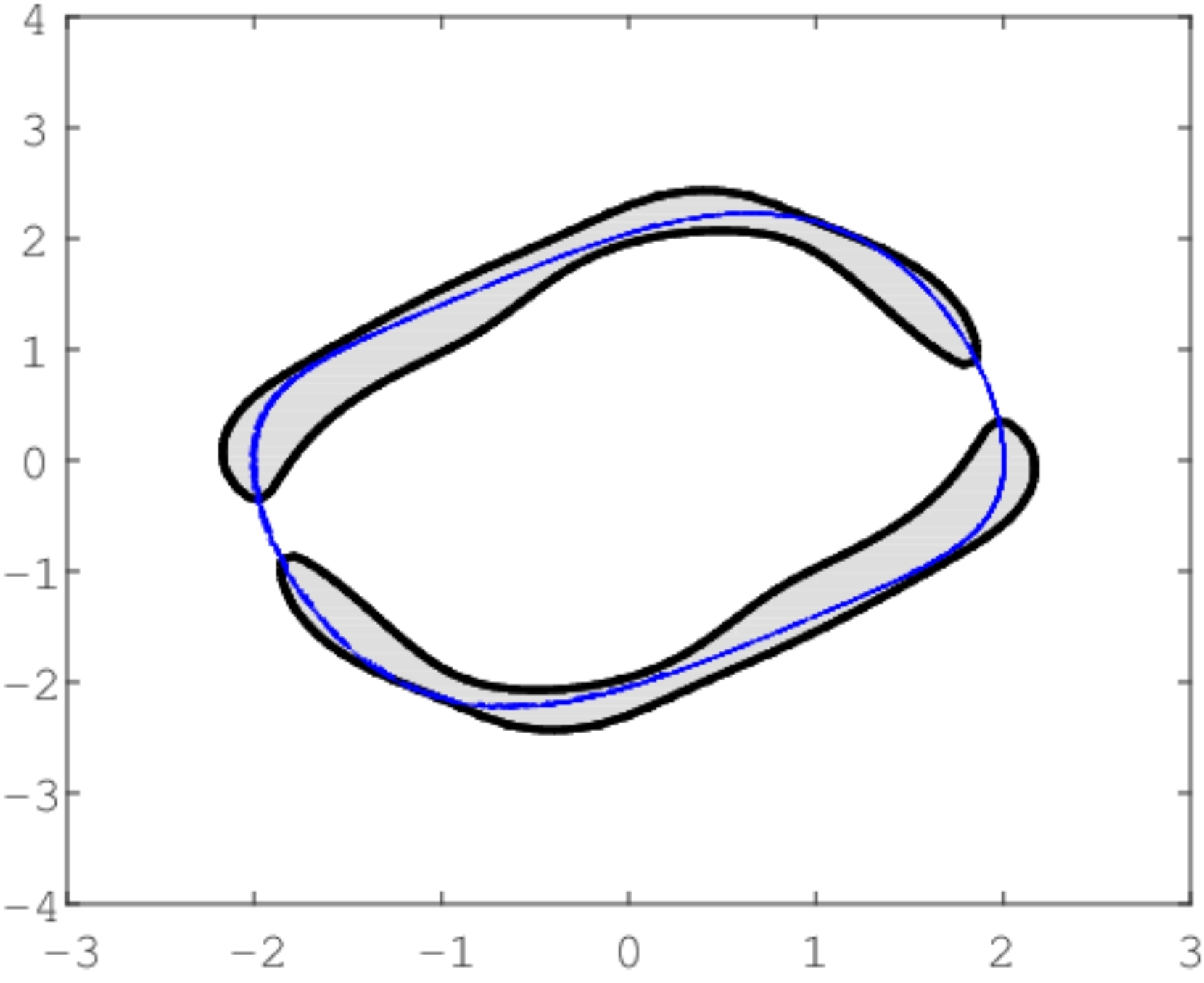}}
\subfigure[$r=8$]{
\includegraphics[scale=\sizesmallfig]{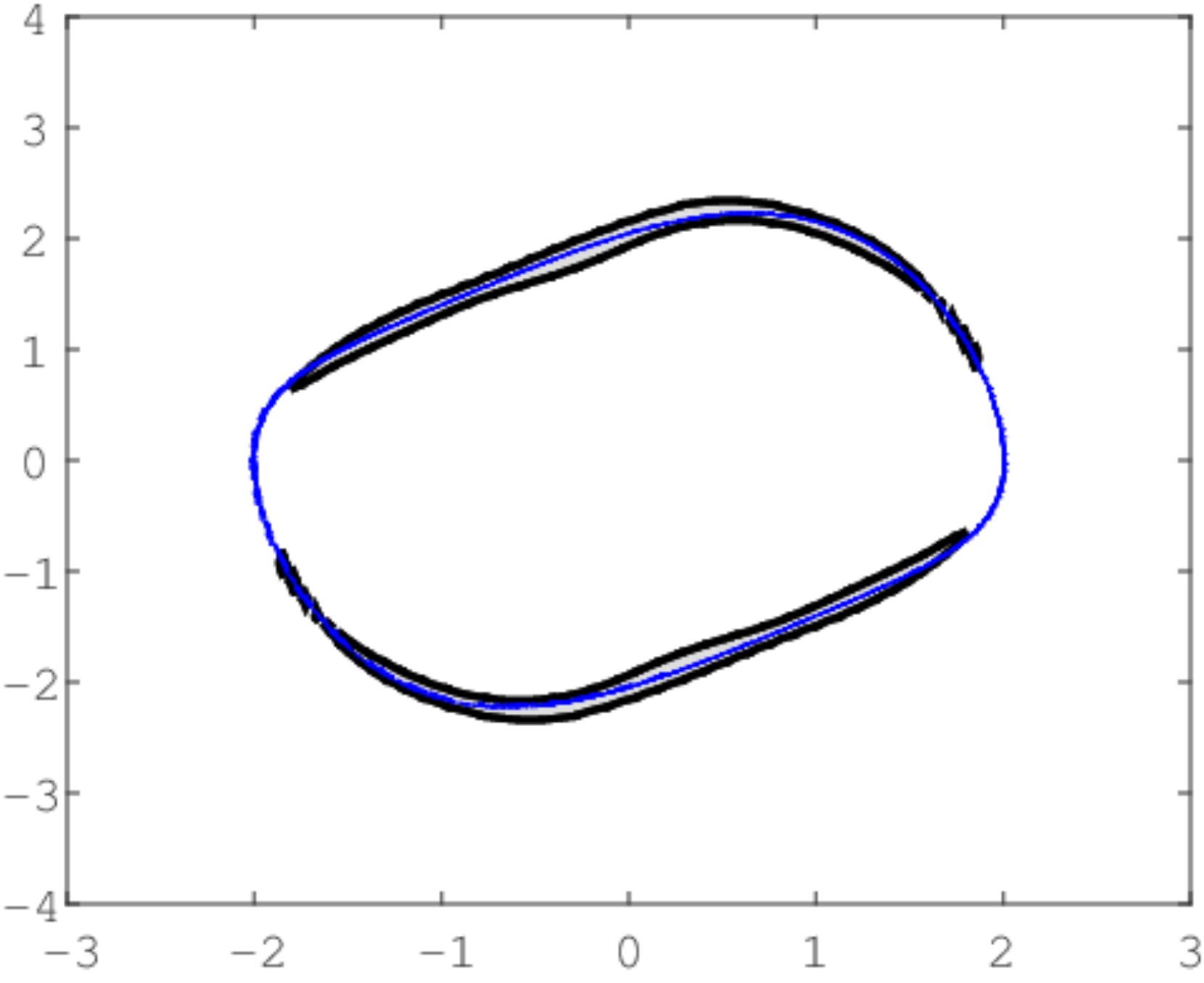}}
\caption{Van der Pol attractor (blue) and approximations $\S^r$ (light gray) for the support of the invariant measure w.r.t.~the map from Example~\ref{ex:vanderpol} for $r \in \{4,6,8\}$ and $a = 0.5$.}
\label{fig:vanderpol}
\end{figure}

Figure~\ref{fig:vanderpol} shows set approximations $\S^r$ {obtained from~\eqref{eq:Sr}} of the support of the measure invariant w.r.t.~the Van der Pol map. As for Example~\ref{ex:henon}, we also represent the ``true'' limit cycle after performing a numerical integration of the Van der Pol system from $t_0 = 0$ to $T = 20$ with random sampled initial conditions within the disk of radius $0.1$ and center $[1, -1]$. This numerical approximation is done with the \texttt{ode45} procedure available inside {\sc Matlab}. Once again, the plots exhibit a quite fast convergence behavior of the approximations $\S^r$ of the invariant measure support to the limit cycle when $r$ increases.
\subsubsection{Arneodo-Coullet System}
\label{ex:arneodo}
Finally, we investigate the Arneodo-Coullet system~\cite{Arneodo85}, representing the dynamics of a forced oscillator. This oscillator can be described by the following three-dimensional time-continuous system:
\begin{align*}
\dot{x_1} & =  x_2 \,, \\
\dot{x_2} & =  x_3 \,, \\
\dot{x_3} & = -a x_1 - b x_2 - x_3 + c x_1^3 \,.
\end{align*}
with general state constraints within the box $\X := [-4, 4] \times [-8, 8] \times [-12, 12]$. 
\begin{figure}[!ht]
\centering
\subfigure[Arneodo-Coullet attractor (blue)]{
\includegraphics[scale=\sizefig]{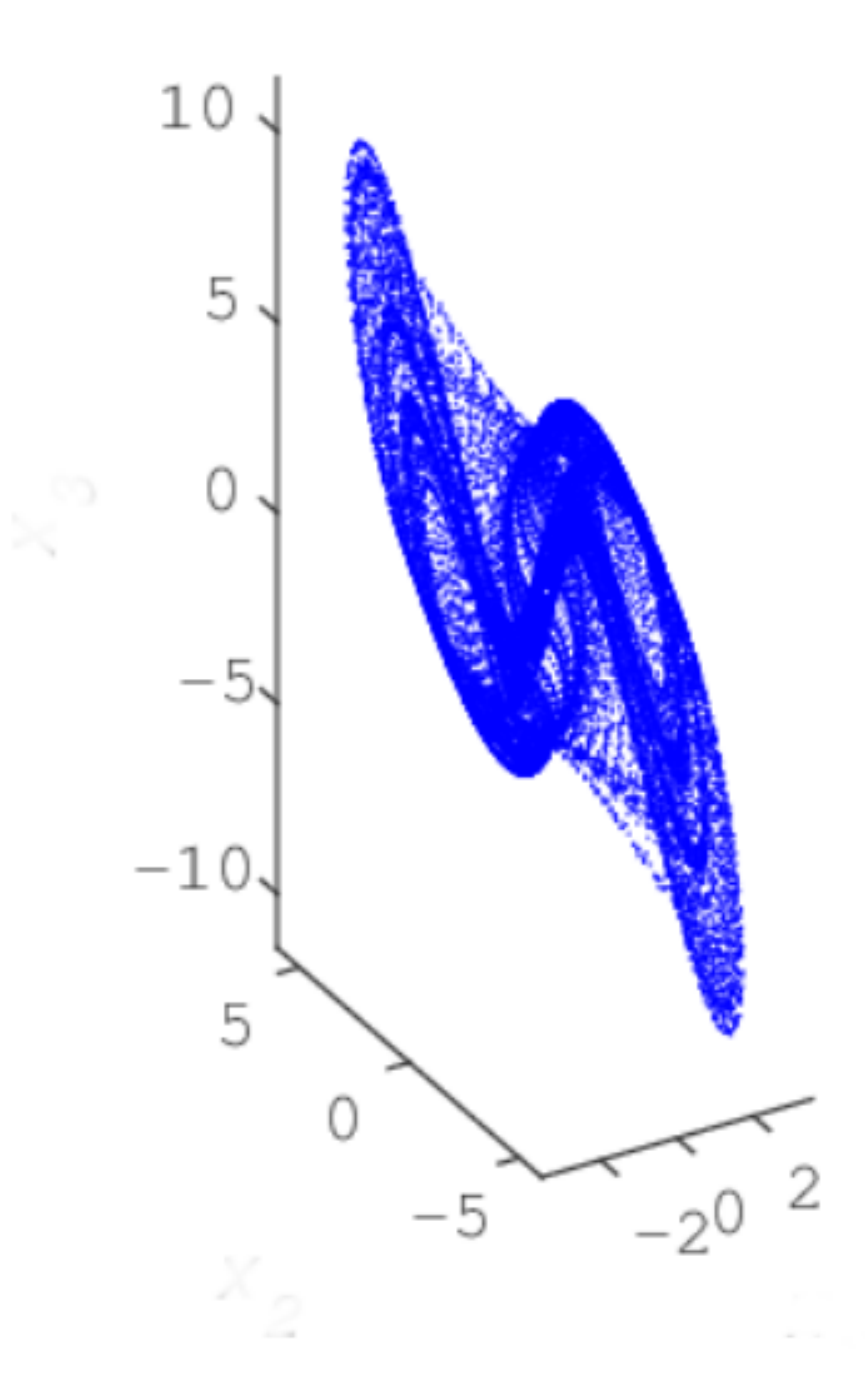}}
\subfigure[$\S^4$ (red)]{
\includegraphics[scale=\sizefig]{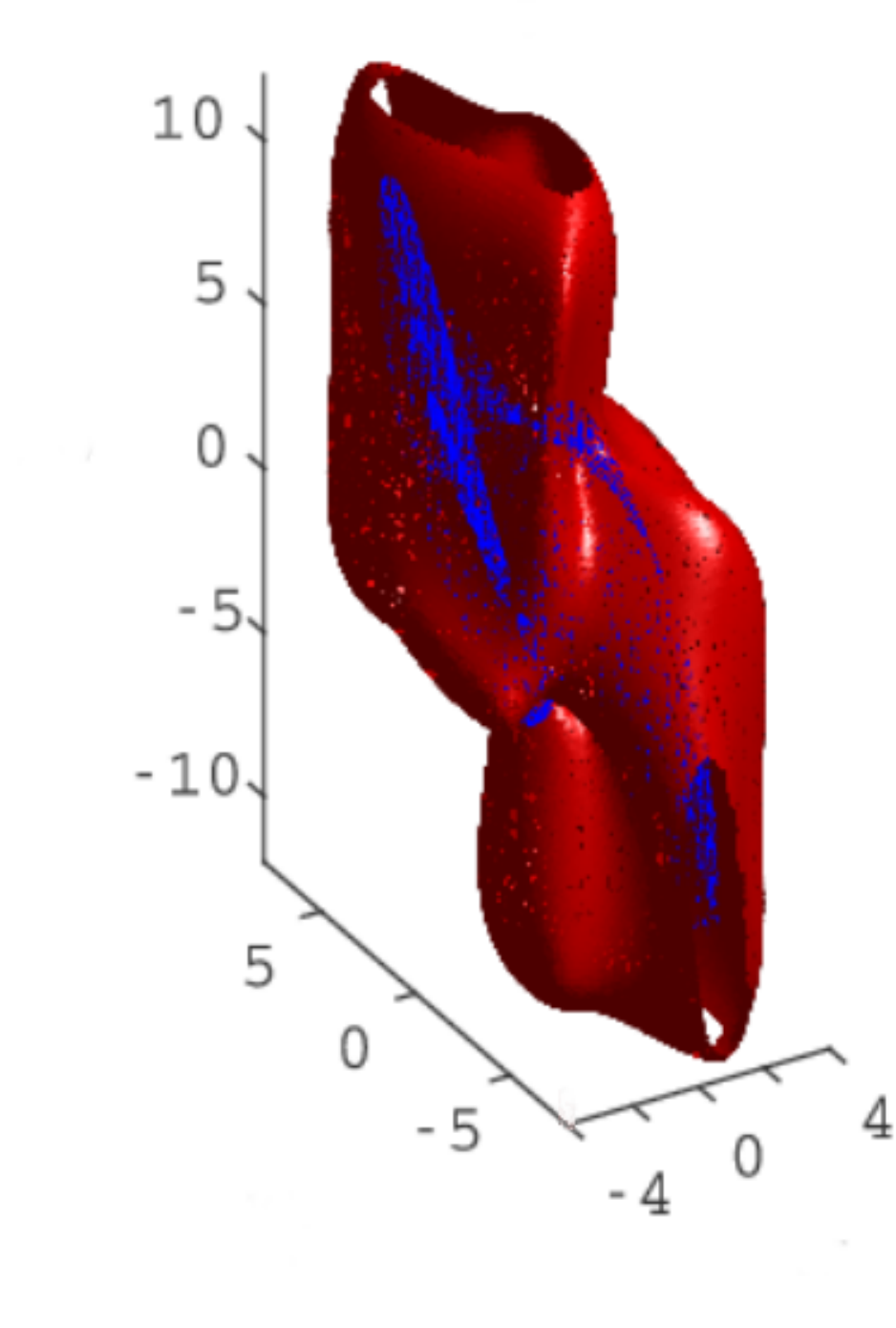}}
\caption{Arneodo-Coullet attractor (blue)
and approximations $\S^4$ (red) for the support of the invariant measure w.r.t.~the map from Example~\ref{ex:arneodo} for $a =-5.5$, $b = 3.5$ and $c = -1$.}
\label{fig:arneodo}
\end{figure}

This system exhibits a chaotic behavior for certain parameter values of $a$, $b$ and $c$, for instance with $a =-5.5$, $b = 3.5$ and $c = -1$.

As for Example~\ref{ex:vanderpol}, we represent the ``true'' attractor after performing numerical integration with the \texttt{ode45} procedure, from $t_0 = 0$ to $T = 1000$ with random sampled initial conditions within the disk of radius $0.1$ and center $[1, 1, 1]$. On Figure~\ref{fig:arneodo}, we observe that the approximation $\S^4$ {obtained from~\eqref{eq:Sr}} provides a reasonably good estimate of the chaotic attractor. The results obtained for higher relaxation orders were not satisfying because of the ill-conditioning of the moment matrix.
\section{Conclusion and Perspectives}
\label{sec:end}
We can summarize the contributions of this article as follows:
\begin{itemize}
\item We propose two methods to characterize invariant measures for discrete-time and continuous-time systems with polynomial dynamics and semialgebraic state constraints. Our approach can also be extended to piecewise-polynomial systems;
\item The first method allows to approximate as close as desired {the moments} of an absolutely continuous invariant measure, under the assumption that this density is square integrable. The second method allows to approximate as close as desired the support of a singular invariant measure by using sublevel sets of the Christoffel polynomial constructed from the moment matrix;
\item Each method relies on solving a hierarchy of finite-dimensional semidefinite programs. While the convergence of the hierarchy is guaranteed in theory, each program can be solved  in practice, thanks to public-available solvers.
\end{itemize}
Our two methods for recovering invariant densities or singular invariant measure supports both rely on an extension of Lasserre's hierarchy of semidefinite relaxations, initially introduced in the context of polynomial optimization.
Numerical experiments show that our two methods already yield fairly good approximations of the moments of invariant measures at modest relaxation orders. One further research direction would be to investigate scaling both methods to large-size systems, when either sparsity or symmetry occurs. 

Regarding the approximation quality of the invariant densities, our first method could produce more satisfactory results by using Chebyshev polynomials or rational function bases as an alternative to the monomials base. 
Regarding the approximation quality of the supports of singular invariant measures, our second method would practically yield better estimates by  handling the issue related to the ill-conditioning of the moment matrix. 

Next, we shall devote research efforts to extend the  uniform convergence properties of the Christoffel polynomial studied in~\cite{Christoffel17} to  certain classes of singular measures. Another track of investigation would be to develop a similar hierarchy of semidefinite programs to study the support of atomic discrete measures corresponding to finite cycles. A first attempt made in~\cite{HenrionFixpoints} consists in setting the objective function of these programs as particular linear moment combinations. However, the question raised  about choosing the adequate objective function to recover a given  finite cycle still remains open.

\paragraph{Acknowledgements} The authors would like to specially acknowledge the precious help of Milan Korda for providing the example from Section~\ref{ex:cfd} as well as for his feedback and suggestions.

\bibliographystyle{plain}

\end{document}